\documentclass[11pt]{article}
\usepackage{amsmath,amsfonts,amsthm,amssymb,mathtools,enumerate}
\usepackage{mathrsfs, graphicx,color,latexsym,cite}
\usepackage[colorlinks=true,
linkcolor=blue,citecolor=blue,
urlcolor=blue]{hyperref}
\newtheorem{theorem}{Theorem}[section]
\newtheorem{lemma}{Lemma}[section]
\newtheorem{problem}{Problem}[section]

\newtheorem{corollary}{Corollary}[section]
\newtheorem{remark}{Remark}[section]

\newtheorem{claim}{Claim}

\newcommand{\SD}{\mathrm{SD}_n}
\newcommand{\SDG}{\mathrm{SD}(n,R,T)}
\newcommand{\PSD}{\mathrm{PSD}_n}
\newcommand{\PSDG}{\mathrm{PSD}(n,R,T)}
\newcommand{\Cay}{\mathrm{Cay}}

\renewcommand\proofname{\bf{Proof}}

\title{\bf \Large Distance-regular Cayley graphs over (pseudo-) semi-dihedral groups}

\author{Xueyi Huang$^a$, \ \ Lu Lu$^{b,}$\thanks{Corresponding author.}\setcounter{footnote}{-1}\footnote{\emph{E-mail address:}  huangxymath@163.com (X. Huang), lulumath@csu.edu.cn (L. Lu), zhanxfmath@163.com (X. Zhan).},    \ \ Xiongfeng Zhan$^a$\\[2mm]
\small $^a$School of Mathematics, East China University of Science and Technology,\\
\small Shanghai, 200237, P. R. China\\
\small $^b$School of Mathematics and Statistics, Central South University,\\ 
\small Changsha, Hunan, 410083, P. R. China
}

\date{ }
\begin{document}
\maketitle\begin{abstract}
Distance-regular graphs are a class of regualr graphs with pretty combinatorial symmetry. In 2007, Miklavi\v{c} and Poto\v{c}nik proposed the problem of charaterizing distance-regular Cayley graphs, which can be viewed as a natural extension of  the problem of characterizing  strongly-regular Cayley graphs (or equivalently, regular partial  difference sets). In this paper, we provide a partial characterization for distance-regular Cayley graphs over semi-dihedral groups and pseudo-semi-dihedral groups, both of which are $2$-groups with a cyclic subgroup of index $2$. 
\par\vspace{2mm}
\par\vspace{1mm}

\end{abstract}

\baselineskip=0.202in

\section{Introduction}\label{sec::1}
In graph theory, distance-regular graphs are a class of regular graphs with pretty combinatorial symmetry. Roughly speaking, a connected graph $X$ is distance-regular, if for every vertex $x$ of $X$, the distance partition of $X$ with respect to $x$ is an equitable partition, and all these equitable partitions share the same quotient matrix. Although this condition
is purely combinatorial, the notion of distance-regular graphs plays a key role in the study of design theory and coding theory, and is closely linked to some other subjects such as finite group theory, finite geometry, representation theory, and association schemes\cite{DKT16}.

In the theory of distance-regular graphs, how to characterize or construct distance-regular graphs of specific types or  parameters is an essential problem. Cayley graphs, a class of vertex-transitive graphs defined by groups and their subsets, might be  good candidates for distance-regular graphs. This is because most of the known distance-regular graphs are vertex-transitive \cite{DK06}, and many infinite families of distance-regular graphs of diameter $2$, namely strongly regular graphs, are constructed from Cayley graphs \cite{FMX15,FX12,GXY13,LM95,LM05,M84,M94,M89,Mom13,Mom14,MX14,MX18,MX22}.

Given a finite group $G$ with identity $1$ and an inverse closed subset $S$ of $G$ with $1\not\in S$, the \textit{Cayley graph} $\mathrm{Cay}(G,S)$ is the graph with vertex set $G$, and with an edge joining two vertices $g,h\in G$ if and only if $g^{-1}h\in S$. It is known that $\mathrm{Cay}(G,S)$ is connected if and only if  $\langle S \rangle=G$, and that $G$ acts regularly on the vertex set of $\mathrm{Cay}(G,S)$ by left multiplicity. In 2007, Miklavi\v{c} and Poto\v{c}nik  \cite{MP07} (see also \cite[Problem 71]{DKT16})  proposed the problem of characterizing distance-regular Cayley graphs.

	\begin{problem}\label{prob::main}
		For a class of groups $\mathcal{G}$, determine all distance-regular graphs, which are Cayley graphs on a group in $\mathcal{G}$.
	\end{problem}

As an early effort on studying Problem \ref{prob::main},  Miklavi\v{c} and Poto\v{c}nik \cite{MP03} determined all distance-regular Cayley graphs over cyclic groups (called \textit{circulants}) by using the method of Schur ring.

\begin{theorem}[{\cite[Theorem 1.2, Corollary 3.7]{MP03}}] \label{thm::cir}
		Let $X$ be a circulant on $n$ vertices. Then $X$ is distance-regular if and only if it is isomorphic to one of the following graphs:
		\begin{enumerate}[(i)]\setlength{\itemsep}{0pt}
			\item the cycle $C_n$;
			\item the complete graph $K_n$;
			\item the complete multipartite graph $K_{t\times m}$, where $tm=n$;
			\item the complete bipartite graph without a perfect matching $K_{m,m}-mK_2$, where $2m = n$ and $m$ is odd;
			\item the Paley graph $P(n)$, where $n\equiv 1\pmod 4$  is prime.
		\end{enumerate}
		In particular, $X$ is a primitive distance-regular graph if and only if $X\cong K_n$, or $n$ is prime, and $X\cong C_n$ or $P(n)$.
	\end{theorem}

Furthermore, by using  Schur ring and Fourier transformation, Miklavi\v{c} and Poto\v{c}nik \cite{MP07} obtained a  classification of distance-regular Cayley graphs over dihedral groups in terms of difference sets. A \textit{$(n,k,\lambda)$-difference set} in a group $G$ of order $n$ is a $k$-subset $D$ of $G$ such that every $g\in G\setminus\{1\}$ has exactly $\lambda$ representations $g=d_1d_2^{-1}$ with $d_1,d_2\in D$. If $k\notin\{n, n-1, 1, 0\}$, then $D$ is \textit{non-trivial}.  A \textit{dihedral group} of order $2n$ is defined as 
$$
D_n=\{\rho,\tau\mid \rho^n=\tau^2=1,\tau\rho\tau=\rho^{-1}\}.
$$ 
Cayley graphs over dihedral groups are called \textit{dihedrants}. A dihedrant of order $2n$ is  denoted by  $\mathrm{Dih}(n,R,T):=\Cay(D_n,\rho^R\cup \rho^T\tau)$, where $R=-R\subseteq \mathbb{Z}_n\setminus\{0\}$, $T\subseteq \mathbb{Z}_n$, and $\rho^I=\{\rho^i\mid i\in I\}$ for $I\in\{R,T\}$. 	
	
\begin{theorem}[{\cite[Theorem 1.3]{MP07}}]\label{thm::dih}
 Let $X$ be a dihedrant on $2n$ vertices other than the cycle $C_{2n}$, the complete graph $K_{2n}$, the complete multipartite graph $K_{t\times m}$, where $tm = 2n$, or the complete bipartite graph without a perfect matching $K_{n,n} - nK_2$. Then $X$ is distance-regular if and only if one of the following holds:
 \begin{enumerate}[(i)]
     \item $X\cong \mathrm{Dih}(n,\emptyset,T)$, where $T$ is a non-trivial difference set in the group $\mathbb{Z}_n$.

\item $n$ is even and $X \cong \mathrm{Dih}(n, R, T)$, where $R$ and $T$ are non-empty subsets of $2\mathbb{Z}_n+1$ such that $\rho^{-1+R}\cup \rho^{-1+T}\tau$ is a non-trivial difference set in the dihedral group $\langle \rho^2,\tau\rangle$ of order $n$.
 \end{enumerate}
If either (i) or (ii) holds, then $X$ is bipartite, non-antipodal, and has diameter 3.
\end{theorem}

By using elementary group theory and structural analysis,  Miklavi\v{c} and \v{S}parl \cite{MS14,MS20}   determined all distance-regular Cayley graphs  over abelian groups and generalized dihedral groups under the condition that the connection set is minimal with respect to some element. Among other things, Abdollahi, van Dam and Jazaeri \cite{ADJ17} determined all distance-regular Cayley graphs of diameter at most three with least eigenvalue $-2$, and van Dam and Jazaeri \cite{DJ19,DJ21} determined some distance-regular Cayley graphs with small valency and provided some characterizations for  bipartite distance-regular Cayley graphs with diameter  $3$ or $4$. For more results on  distance-regular Cayley graphs, we refer the reader to  \cite{HD22,HDL23,ZLH23}.

According to the former literatures, it is not hard to see that   $p$-groups, especially elementary abelian $p$-groups, play a key role in the construction of strongly regular Cayley graphs \cite{FMX15,FX12,Mom14,MX22}. Thus it is natural to ask for a characterization of distance-regular Cayley graphs over  $p$-groups. In this paper, we focus on the characterization of distance-regular Cayley graphs over semi-dihedral groups and pseudo-semi-dihedral groups. For $r\ge 2$ and $n=2^r$, the \textit{semi-dihedral group} and \textit{pseudo-semi-dihedral group}  of order $4n$ are defined as
\[
\SD=\langle\rho,\tau\mid\rho^{2n}=\tau^2=1,\tau\rho\tau=\rho^{n-1} \rangle
\]
and 
\[
\PSD=\langle\rho,\tau\mid\rho^{2n}=\tau^2=1,\tau\rho\tau=\rho^{n+1} \rangle,
\]
respectively. Note that both  $\SD$ and $\PSD$ are $2$-groups with a cyclic subgroup of index $2$.  Cayley graphs over semi-dihedral groups and pseudo-semi-dihedral groups are called 
\textit{semi-dihedrants} and \textit{pseudo-semi-dihedrants}, respectively. Then it is easy to see that a semi-dihedrant (resp. pseudo-semi-dihedrant) of order $4n$ is of the form  $\SDG:=\mathrm{Cay}(\SD,\rho^R\cup \rho^T\tau)$ (resp. $\PSDG:=\mathrm{Cay}(\PSD,\rho^R\cup \rho^T\tau)$), where $R=-R\subseteq \mathbb{Z}_{2n}\setminus\{0\}$ and $T=(n+1)T\subseteq \mathbb{Z}_{2n}$ (resp. $T=(n-1)T\subseteq \mathbb{Z}_{2n}$). Our main results are as follows.

\begin{theorem}\label{thm::main1}
Let $X=\mathrm{SD}(n,R,T)$ be a semi-dihedrant with  $n=2^r>4$. Then $X$ is distance-regular if and only if it is isomorphic to one of the following graphs:
\begin{enumerate}[(i)]\setlength{\itemsep}{0pt}
\item the complete graph $K_{4n}$;
\item the complete multipartite graph $K_{s\times t}$  with $st=4n$;

\item the complete bipartite graph without a perfect matching $K_{2n,2n} - 2nK_2$;

\item the graph $\mathrm{SD}(n,\emptyset,T)$, where $T=(n+1)T$ is  a non-trivial difference set in the group $\mathbb{Z}_{2n}$;
\item the graph $\SDG$, where $R=-R$ and  $T=n+T$ are non-empty subsets of $2\mathbb{Z}_{2n}+1$ such that $\rho^{-1+R}\cup\rho^{-1+T}\tau$ is a non-trivial difference set in the group $\langle\rho^2,\tau\rangle$;
\item the graph $\SDG$,   where $R=-R\subseteq 2\mathbb{Z}_{2n}+1$ and $T\subseteq 2\mathbb{Z}_{2n}$ are non-empty subsets  such that $\rho^{-1+R}\cup\rho^{-1+T}\tau$ is a non-trivial difference set in the group $\langle\rho^2,\rho\tau\rangle$;
\item the graph $\SDG$, where $R=-R\subseteq 2\mathbb{Z}_{2n}+1$ and $T\subseteq 2\mathbb{Z}_{2n}$ are subsets of size $n/2$ such that  $R\cap (n+R)=T\cap (n+T)=\emptyset$ and $|R\cap (i+R)|+|T\cap (i+T)|=n/2$ for all $i\in 2\mathbb{Z}_{2n}\setminus\{0,n\}$.
\end{enumerate}
In particular,  the graphs in (iv)--(vi) are  non-antipodal bipartite non-trivial distance-regular graph with  diameter $3$, and the graph in  (vii) is a $2$-fold antipodal bipartite non-trivial distance-regular graph with diameter $4$ (i.e., Hadamard graph).
\end{theorem}

\begin{theorem}\label{thm::main2}
Let $X=\PSDG$ be a pseudo-semi-dihedrant  with $n=2^r>4$. Then $X$ is distance-regular if and only if it is isomorphic to one of the following graphs:
\begin{enumerate}[(i)]\setlength{\itemsep}{0pt}
\item the complete graph $K_{4n}$;
\item the complete multipartite graph $K_{s\times t}$  with $st=4n$;

\item the complete bipartite graph without a perfect matching $K_{2n,2n} - 2nK_2$;
\item the graph $\mathrm{PSD}(n,\emptyset,T)$, where $T=(n-1)T$ is  a non-trivial difference set in the group $\mathbb{Z}_{2n}$;
\item the graph $\PSDG$, where $R=-R$ and $T=n-T$ are non-empty subsets of $2\mathbb{Z}_{2n}+1$ such that $(-1+R,0)\cup(-1+T,1)$ is a non-trivial difference set in the group $2\mathbb{Z}_{2n}\oplus\mathbb{Z}_2$; 
\item the graph $\PSDG$, where $R=-R\subseteq 2\mathbb{Z}_{2n}+1$ and $T=-T\subseteq 2\mathbb{Z}_{2n}$ are non-empty subsets  such that $\rho^{-1+R}\cup\rho^{-1+T}\tau$ is a non-trivial difference set in the group $\langle\rho^2,\rho\tau\rangle=\langle \rho\tau\rangle$;
\item the graph $\PSDG$, where 
$R=-R$ and $T=n-T$ are subsets of $2\mathbb{Z}_{2n}+1$ of size $n/2$ such that  $R\cap (n+R)=T\cap (n+T)=\emptyset$ and $|R\cap (i+R)|+|T\cap (i+T)|=n/2$ for all $i\in 2\mathbb{Z}_{2n}\setminus\{0,n\}$.
\end{enumerate}
In particular, the graphs in (iv)--(vi) are  non-antipodal bipartite non-trivial distance-regular graphs with  diameter $3$, and the graph in  (vii) is a $2$-fold antipodal bipartite non-trivial distance-regular graph with diameter $4$ (i.e., Hadamard graph).
\end{theorem}

\section{Preliminaries}\label{sec::2}

Let $X$ be a connected graph with vertex set $V(X)$ and edge set $E(X)$. For two vertices $x,y\in V(X)$, denote by $d_X(x,y)$ the {\it distance} between $x$ and $y$, which is the length of a shortest path from $x$ to $y$. The {\it diameter} $d:=d(X)$ is the maximum distance between pairs of vertices in $X$. For $0\le i\le d$, let $N_i(x)=\{y\in V(X)\mid d_X(x,y)=i\}$. We write $N(x)$ for $N_1(x)$ to denote the neighborhood of $x$. The graph $X$ is called {\it distance-regular} if, for all integers $i,j,k$ with $0\le i,j,k\le d$ and all $x,y\in V(X)$ with $d_X(x,y)=k$, the number
\[p_{i,j}^k=|\{z\in V(X)\mid d_X(x,z)=i,d_X(y,z)=j\}|\]
is independent of the choice of $x$ and $y$. The constants $p_{i,j}^k$ ($0\le i,j,k\le d$) are known as the {\it intersection numbers} of $X$. For  convenience, we define $c_i=p^i_{1,i-1}$ ($1\le i\le d$), $a_i=p_{1,i}^i$ ($0\le i\le d$), $b_i=p_{1,i+1}^i$ ($0\le i\le d-1$), $k_i=p_{i,i}^0$ ($0\le i\le d$), and set $c_0=b_d=0$. Observe that $a_0=0$ and $c_1=1$. Furthermore, $a_i+b_i+c_i=k$ ($0\le i\le d$), where $k=k_1$. Following convention, we abbreviate $\lambda=a_1$ and $\mu=c_2$. It is clear that $k_i=|N_i(v)|$ for any $v\in V(X)$. The array
$\{b_0,b_1,\ldots,b_{d-1};c_1,c_2,\ldots,c_d\}$
is called the {\it intersection array} of $X$. A distance-regular graph on $n$ vertices with valency $k$ and diameter $2$ is also called a  \textit{strongly regular graph} with parameters $(n,k,\lambda=a_1,\mu=c_2)$.

Let $X$ be a graph, and let $\mathcal{B}=\{B_1,\ldots,B_\ell\}$ be a partition of $V(X)$ (here $B_i$ are called \textit{blocks}). The \textit{quotient graph} of $X$ with respect to $\mathcal{B}$, denoted by $X_\mathcal{B}$, is the graph with vertex set $\mathcal{B}$, and with  $B_i,B_j$ ($i\neq j$) adjacent if and only if there exists at least one edge between  $B_i$ and $B_j$ in $X$. Moreover, we say that $\mathcal{B}$ is an \textit{equitable partition} of $X$ if there are integers $b_{ij}$ ($1\leq i,j\leq \ell$) such that every vertex in $B_i$ has exactly $b_{ij}$ neighbors in $B_j$. In particular, if every block of $\mathcal{B}$ is an independent set, and between any two blocks there are either no edges or there is a perfect matching, then $\mathcal{B}$ is an equitable partition of  $X$. In this situation,  $X$ is called a \textit{cover} of its quotient graph $X_\mathcal{B}$, and the blocks are called \textit{fibres}. If $X_\mathcal{B}$ is connected, then all fibres have the same size, say $r$, called \textit{covering index}. A graph $X$ of diameter $d$ is \textit{antipodal} if the relation $\mathcal{R}$ on $V(X)$ defined by $u\mathcal{R}v\Leftrightarrow d_X(u,v)\in\{0,d\}$ is an equivalence relation, and the corresponding equivalence classes are called \textit{antipodal classes}. A cover of index $r$, in which the fibres are antipodal classes, is called \textit{an $r$-fold  antipodal cover} of its quotient.

	Suppose that $X$ is a distance-regular graph with diameter $d$. For  $i\in\{1,\ldots,d\}$, the \textit{$i$-th distance graph} $X_i$  is  the graph with vertex set  $V(X)$ in which two distinct vertices are adjacent if and only if they are at distance $i$ in $X$. If, for any $1\le i\le d$, $X_i$ is connected, then $X$ is {\it primitive}. Otherwise, $X$ is {\it imprimitive}.  It is  known that an imprimitive distance-regular graph with valency at least $3$ is either bipartite, antipodal, or both \cite[Theorem 4.2.1]{BCN89}. If $X$ is a bipartite distance-regular graph, then $X_2$ has two connected components, which are called the \textit{halved graphs} of $X$ and denoted by $X^+$ and $X^-$. For convenience, we use  $\frac{1}{2}X$ to represent any one of these two graphs.  If $X$ is an antipodal distance-regular graph, then all antipodal classes have the same size, say $r$, and form an equitable partition $\mathcal{B}^\ast$ of $X$. The quotient graph $\overline{X}:=X_{\mathcal{B}^\ast}$ is called the \textit{antipodal quotient} of $X$. If $d=2$, then $X$ is a complete multipartite graph. If $d\geq 3$, then the edges between  two distinct antipodal classes of $X$ form an empty set or a perfect matching. Thus $X$ is an $r$-fold antipodal  cover of $\overline{X}$ with the antipodal classes as its fibres.

\begin{lemma}[{\cite[Proposition 4.2.2]{BCN89}}] \label{lem::2.1}
		Let $X$ denote an imprimitive distance-regular graph with diameter $d$ and valency $k \geq 3$. Then the following statements hold.
		\begin{enumerate}[(i)]\setlength{\itemsep}{0pt}
			\item If $X$ is bipartite, then the halved graphs of $X$ are non-bipartite distance-regular graphs with diameter $\lfloor\frac{d}{2}\rfloor$.
			\item If $X$  is antipodal, then $\overline{X}$ is a distance-regular graph with diameter $\lfloor\frac{d}{2}\rfloor$.
			\item  If $X$ is antipodal, then $\overline{X}$ is not antipodal, except when $d\leq 3$  (in that case $\overline{X}$  is a complete graph), or when $X$ is bipartite with $d = 4$ (in that case $\overline{X}$ is a complete bipartite graph).
			
			\item If $X$ is antipodal and has odd diameter or is not bipartite, then $\overline{X}$ is primitive.
			\item If $X$ is bipartite and has odd diameter or is not antipodal, then the halved graphs of $X$ are primitive.
			\item If $X$ has even diameter and is both bipartite and antipodal, then $\overline{X}$ is bipartite. Moreover, if $\frac{1}{2}X$ is a halved graph of $X$, then it is antipodal, and $\overline{\frac{1}{2}X}$ is primitive and isomorphic to $\frac{1}{2}\overline{X}$.
		\end{enumerate}
	\end{lemma}

Let $G$ be a transitive permutation group acting on a set $\Omega$. An \textit{imprimitivity system}  for $G$ is a partition $\mathcal{B}$ of $\Omega$ which is invariant under the action of $G$, i.e., for every block $B\in \mathcal{B}$ and for every $g\in G$, we have either $B^g=B$ or $B^g\cap B=\emptyset$.
	
\begin{lemma}[{\cite[Theorem 6.2]{GH92}}]\label{lem::2.2}
		Let $X$ be an antipodal distance-regular graph with diameter $d\geq 3$, and let $\mathcal{B}$ be an equitable partition of $X$ with each block contained in a
		fibre of $X$. Assume that no block of $\mathcal{B}$ is a single vertex, or a fibre. Then all
		blocks  of  $\mathcal{B}$ have the same size, and the  quotient graph $X_{\mathcal{B}}$ is an
		antipodal distance-regular graph with diameter $d$. Moreover, $X$ and $X_{\mathcal{B}}$ have isomorphic antipodal quotients.
	\end{lemma}

	\begin{lemma}[{\cite[p. 425, p. 431]{BCN89}}]\label{lem::2.3}
Let $X$ be an $r$-fold antipodal distance-regular graph on $n$ vertices with  diameter $d$ and valency $k$.
\begin{enumerate}[(i)]\setlength{\itemsep}{0pt}
\item If $X$ is non-bipartite and $d=3$, then $n=r(k+1)$, $k=\mu(r-1)+\lambda+1$, and $X$ has  the intersection array $\{k,\mu(r-1),1;1,\mu,k\}$ and the spectrum $\{k^1,\theta_1^{m_1},\theta_2^k,\theta_3^{m_3}\}$, where
\begin{equation*}
\theta_1=\frac{\lambda-\mu}{2}+\delta,~~\theta_2=-1,~~\theta_3=\frac{\lambda-\mu}{2}-\delta,~~\delta=\sqrt{k+\left(\frac{\lambda-\mu}{2}\right)^2},
\end{equation*}
and
\begin{equation*}
m_1=-\frac{\theta_3}{\theta_1-\theta_3}(r-1)(k+1),~~m_3=\frac{\theta_1}{\theta_1-\theta_3}(r-1)(k+1).
\end{equation*}
Moreover, if $\lambda\neq\mu$, then all eigenvalues  of $X$ are integers.

\item If $X$ is bipartite and $d=4$, then $n=2r^2\mu$, $k=r\mu$, and $X$ has the intersection array $\{r\mu, r\mu-1,(r-1)\mu, 1;1,\mu,r\mu-1,r\mu\}$.
\end{enumerate}
\end{lemma}

A strongly regular graph with parameters $(n,k=\frac{n-1}{2},\lambda=\frac{n-5}{4},\mu=\frac{n-1}{4})$  ($n\equiv 1\pmod 4$) is called a \textit{conference graph}.  Let $\mathbb{F}_q$ denote the finite field of order $q$  where $q\equiv 1\pmod 4$ is a prime power. The \textit{Paley graph} $P(q)$ is  the graph with vertex set $\mathbb{F}_q$ in which two distinct vertices $u,v$ are adjacent if and only if $u-v$ is a square in the multiplicative group of $\mathbb{F}_q$. It is known that Paley graphs are conference graphs \cite{Er63,Sac62}.

\begin{lemma}[{\cite[p. 180]{BCN89}}]\label{lem::2.4}
Let $X$ be a conference graph (or particularly, Paley graph). Then:
\begin{enumerate}[(i)]\setlength{\itemsep}{0pt}
\item $X$ has no distance-regular $r$-fold antipodal covers for $r>1$, except for the pentagon $C_5\cong P(5)$, which is covered by the decagon $C_{10}$;
\item $X$ cannot be a halved graph of a bipartite distance-regular graph.
\end{enumerate}
\end{lemma}

A \textit{dicyclic group} of order $4n$ is defined as 
$$
\mathrm{Dic}_n=\langle\rho,\tau\mid \rho^{2n}=1,\tau^2=\rho^n,\tau^{-1}\rho\tau=\rho^{-1}\rangle.
$$ 
Cayley graphs over dicyclic groups are called \textit{dicirculants}. A dicirculant of order $4n$ can be denoted  by  $\mathrm{Dic}(n,R,T):=\Cay(D_n,\rho^R\cup \rho^T\tau)$, where $R=-R\subseteq \mathbb{Z}_{2n}\setminus\{0\}$ and $T=n+T\subseteq \mathbb{Z}_{2n}$. In \cite{HDL23}, Huang, Das and Lu obtained a partial characterization of  distance-regular dicirculants.

\begin{lemma}[{\cite[Theorem 1.1]{HDL23}}]\label{lem::2.5}
Let $X$ be a dicirculant on $4n$ vertices. Then $X$ is distance-regular if and only if it is isomorphic to one of the following graphs:
\begin{enumerate}[(i)]\setlength{\itemsep}{0pt}
\item the complete graph $K_{4n}$;
\item the complete multipartite graph $K_{t\times m}$  with $tm=4n$;
\item the graph $\mathrm{Dic}(n,R,T)$ for even $n$, where $R=-R$ and $T=n+T$ are non-empty subsets of  $2\mathbb{Z}_{2n}+1$ such that $|R\cap T|<n$ and $|R\cap (i+R)|+|T\cap (i+T)|=2|(j+R)\cap T|$ for all $i,j\in 2\mathbb{Z}_{2n}$, $i\neq 0$.
\end{enumerate}
In particular, the graph in  (iii) is a non-antipodal bipartite non-trivial distance-regular graph with  diameter $3$.
\end{lemma}

 The following result is an analogue of the characterization of distance-regular Cayley graphs over dihedrant groups given in \cite{MP07}. 

\begin{lemma}[{\cite[Theorem 4.1]{ZLH23}}]\label{lem::2.6}
Let $X$ be a Cayley graph over $\mathbb{Z}_n\oplus\mathbb{Z}_2$ with $n>2$ being even. Then $X$ is distance-regular if and only if it is isomorphic to one of the following graphs:
\begin{enumerate}[(i)]
\setlength{\itemsep}{0pt}
\item the complete graph $K_{2n}$;
\item the complete multipartite graph $K_{t\times m}$  with $tm=2n$;
\item the complete bipartite graph without a perfect matching $K_{n,n} - nK_2$;
\item the graph $\Cay(\mathbb{Z}_n\oplus\mathbb{Z}_2,(R_0,0)\cup (R_1,1))$, where $R_0=-R_0$ and $R_1=-R_1$ are non-empty subsets of  $2\mathbb{Z}_{n}+1$ such that $(-1+R_0,0)\cup(-1+R_1,1)$ is a non-trivial difference set in $2\mathbb{Z}_n\oplus\mathbb{Z}_2$.
\end{enumerate}
In particular, the graph in  $(iv)$ is a non-antipodal bipartite non-trivial distance-regular graph with  diameter $3$.
\end{lemma}

Let $G$ be a group, and let $\mathbb{Z}G$ be the group algebra of $G$ over the ring of  integers $\mathbb{Z}$. For an integer $m$ and an element $a=\sum_{g\in G} a_g g\in \mathbb{Z}G$, we define
$$a^{(m)}=\sum_{g\in G} a_g g^m\in \mathbb{Z}G.$$
Also, for a subset $T\subseteq G$, we denote 
$$
	\underline{T}=\sum_{g\in T}g\in \mathbb{Z}G.
$$
If there is a  partition $\{T_0,T_1,\ldots,T_r\}$  of $G$ satisfying
	\begin{enumerate}[(i)]
		\item $T_0=\{e\}$,
		\item for any $i\in \{1,\ldots,r\}$, there exists some $j\in \{1,\ldots,r\}$ such that $\underline{T_i}^{(-1)}=\underline{T_j}$,
		\item  for any $i,j\in \{1,\ldots,r\}$, there exist integers $p_{i,j}^k$ ($0\leq k\leq r$) such that $$\underline{T_i}\cdot \underline{T_j}=\sum_{k=0}^rp_{ij}^k \cdot \underline{T_k},$$
	\end{enumerate}
	then the $\mathbb{Z}$-module $\mathcal{S}$ spanned by $\underline{T_0},\underline{T_1},\ldots,\underline{T_r}$ is exactly a subalgebra of $\mathbb{Z}G$, and called a \textit{Schur ring} over $G$. In this situation, the  basis $\{\underline{T_0},\underline{T_1},\ldots,\underline{T_r}\}$ is called the \textit{simple basis} of the Schur ring $\mathcal{S}$. We say that the Schur ring  $\mathcal{S}$ is \textit{primitive} if $\langle T_i\rangle=G$ for every $i\in \{1,\ldots,r\}$. In partitular, if $T_0=\{e\}$ and $T_1=G\setminus \{e\}$, then the Schur ring spanned  by $\underline{T_0}$ and $\underline{T_1}$ is called \textit{trivial}. Note that a trivial Schur ring is primitive.

	Let $X=\Cay(G,S)$ be a Cayley graph of diameter $d$. Denote by
	$$
	\mathcal{N}_i=\{g\in G\mid d_X(g,1)=i\}.
	$$
	The $\mathbb{Z}$-submodule of  $\mathbb{Z}G$ spanned by  $\underline{\mathcal{N}_0},\underline{\mathcal{N}_1},\ldots,\underline{\mathcal{N}_d}$ is called the \textit{distance module} of  $X$, and is denoted by \textit{$\mathcal{D}_\mathbb{Z}(G,S)$}. In \cite{MP03},  Miklavi\v{c} and Poto\v{c}nik  provided an algebraic characterization for  (primitive) distance-regular Cayley graphs in terms of the Schur ring and the distance module.
	
	\begin{lemma}[{\cite[Proposition 3.6]{MP03}}] \label{lem::2.7}
		Let $X=\Cay(G,S)$  denote a distance-regular Cayley graph and $\mathcal{D}=\mathcal{D}_\mathbb{Z}(G,S)$ its distance module. Then:
		\begin{enumerate}[(i)]\setlength{\itemsep}{0pt}
			\item $\mathcal{D}$ is a (primitive) Schur ring over $G$ if and only if $X$ is a (primitive) distance-regular graph;
			\item $\mathcal{D}$ is the trivial Schur ring  over $G$ if and only if $X$ is isomorphic to the complete graph.
		\end{enumerate}
	\end{lemma}

	\begin{lemma}[\cite{Eno69,KE68}]\label{lem::2.8}
For every $n=2^r>4$, there are no non-trivial primitive Schur rings over the semi-dihedral group $\SD$ and the pseudo-semi-dihedral group $\PSD$.
\end{lemma}

If $X$ is a primitive distance-regular Cayley graph over $\SD$ (resp. $\PSD$)  with $n=2^r >4$, then its distance module would be a primitive Schur ring over $\SD$ (resp. $\PSD$) by Lemma \ref{lem::2.7} (i), and hence must be  the trivial Schur ring  by Lemma \ref{lem::2.8}. Therefore, by Lemma \ref{lem::2.7} (ii), we obtain the following result.
	
\begin{corollary}\label{cor::2.1}
		Let $X$ be a distance-regular  semi-dihedrant or a distance-regular pseudo-semi-dihedrant of order $4n$ with $n=2^r>4$. If $X$ is primitive, then $X$ is isomorphic to the complete graph $K_{4n}$.
\end{corollary}

\begin{lemma}[{\cite[Lemma 2.2]{MP07}}] \label{lem::2.9}
		Let $X=\Cay(G,S)$ denote a Cayley graph with the group $G$ acting regularly on the vertex set of $X$ by left multiplication. Suppose there exists an imprimitivity system $\mathcal{B}$ for $G$. Then the block $B\in\mathcal{B}$ containing the identity  $1\in G$ is a subgroup in $G$. Moreover,
		\begin{enumerate}[(i)]\setlength{\itemsep}{0pt}
			\item if $B$ is normal in $G$, then $X_\mathcal{B}=\Cay(G/B,S/B)$, where $S/B=\{sB\mid s\in S\setminus B\}$;
			\item if there exists an abelian subgroup $A$ in $G$ such that $G = AB$, then $X_\mathcal{B}$ is isomorphic to a
			Cayley graph on the group $A/(A\cap B)$.
		\end{enumerate}
	\end{lemma}

\begin{lemma}[\cite{MP07}]\label{lem::2.10}
Let $G$ be a group of order $2n$ and $S$ a subset of $G$. Then the following statements are equivalent:
\begin{enumerate}[(i)]
    \item $S\subseteq G\setminus\{1\}$, $S=S^{(-1)}$ and $\Cay(G,S)$ is a bipartite non-trivial distance-regular graph with diamter $3$ and intersection array $\{k,k-1,k-\mu;1,\mu,k\}$;
    \item there is a subgroup $H$ of index $2$ in $G$ such that for every $a\in G\setminus H$, the set $D=a^{-1}S$ is a non-trivial $(n,k,\mu)$-difference set in $H$ satisfying $D^{(-1)}=aDa$.
    \item there are a subgroup $H$ of index $2$ in $G$ and an element $a\in G\setminus H$ such that the set $S=a^{-1}S$ is a non-trivial $(n,k,\mu)$-difference set in $H$ satisfying $D^{(-1)}=aDa$.
\end{enumerate}
Moreover, if (i)--(iii) hold, then $H\setminus \{1\}$ is exactly the set of vertices of $\Cay(G,S)$ which are at distance $2$ from the vertex $1$.
\end{lemma}

Let $H$ be a finite group, and let $N$ be a proper subgroup of $H$ such that $|N|=r$ and $[H:N]=m$. Then a $k$-subset $D$ of $H$ is an $(m,r,k,\mu)$-\textit{relative difference set} relative to $N$ (the `forbidden' subgroup) if and only if $\underline{D}\cdot \underline{D}^{(-1)}=k\cdot1+\mu\cdot \underline{H\setminus N}$. Moreover, we say that $D$ is symmetric whenever $D^{-1}$ is also a relative difference set (possibly with a different forbidden subgroup).  In particular, if  $N$ is normal, then $D$ is symmetric.

\begin{lemma}[{\cite[Proposition 10]{DJ21}}]\label{lem::2.11}
	Let $G$ be a group of order $2r^2\mu$ and $S$ a subset of $G$. Then the following statements are equivalent:
	\begin{enumerate}[(i)]
		\item $S\subseteq G\setminus\{1\}$, $S=S^{(-1)}$ and $\Cay(G,S)$ is an antipodal bipartite distance-regular graph with diameter $4$ and intersection array $\{r\mu,r\mu-1,(r-1)\mu,1;1,\mu,r\mu-1,r\mu\}$;
		\item there is a subgroup $H$ of index $2$ in $G$ and a subgroup $N$ of $H$ of order $r$ such that for every $a\in G\setminus H$, the set $D=a^{-1}S$ is a symmetric $(r\mu,r,r\mu,\mu)$-relative  difference set relative to $N$ in $H$ satisfying $D^{(-1)}=aDa$;
		\item there is a subgroup $H$ of index $2$ in $G$, a subgroup $N$ of $H$ of order $r$ and an element $a\in G\setminus H$ such that the set $D=a^{-1}S$ is a symmetric $(r\mu,r,r\mu,\mu)$-relative  difference set relative to $N$ in $H$ satisfying $D^{(-1)}=aDa$.
	\end{enumerate}
Moreover, if (i)--(iii) hold, then $H\setminus \{1\}$ is exactly the set of vertices of $\Cay(G,S)$ which are at distance $2$ or $4$ from the vertex $1$, and $N\setminus \{1\}$ is exactly the set of vertices of $\Cay(G,S)$ which are at distance $4$ from the vertex $1$.
\end{lemma}

\begin{lemma}[{\cite[Theorem 2.1]{Hir03}}]\label{lem::2.12}
If a group $G$ of order $4n$ contains a $(2n,2,2n,n)$-relative difference set, then a Sylow
$2$-subgroup of $G$ is non-cyclic and $n$ is even unless $n=1$.
\end{lemma}

For a positive integer $n$, let $\mathbb{Z}_n^*$ be the unit group of the ring $\mathbb{Z}_n$, and $\omega$ a fixed primitive $n$-th root of unity. Let $\mathbb{F}=\mathbb{Q}(\omega)$ be the $n$-th cyclotomic field over the rationals. For a subset $A\subseteq \mathbb{Z}_n$, $i\in \mathbb{Z}_n$ and $c\in\mathbb{Z}_n^*$, denote by $cA=\{ca\colon a\in A\}$, $i+A=\{i+a\colon a\in A\}$, and $i-A=i+(-1)A$. Let  $\Delta_A:\mathbb{Z}_n\rightarrow \mathbb{F}$ denote the characteristic function of $A$, that is, $\Delta_A(x)=1$ if $x\in A$, and $\Delta_A(x)=0$ otherwise. Let $\mathbb{F}^{\mathbb{Z}_n}$ be the $\mathbb{F}$-vector space of all functions $f$: $\mathbb{Z}_n\rightarrow \mathbb{F}$ with the scalar multiplication and addition defined pointwise. If we define the multiplication point-wise:
\[(f\cdot g)(z)=f(z)g(z),~f,g\in\mathbb{F}^{\mathbb{Z}_n},\]
then $(\mathbb{F}^{\mathbb{Z}_n},\cdot)$ forms an $\mathbb{F}$-algebra. If we define the multiplication as the \textit{convolution}:
\[(f*g)(z)=\sum_{i\in\mathbb{Z}_n}f(i)g(z-i),~f,g\in\mathbb{F}^{\mathbb{Z}_n},\]
then $(\mathbb{F}^{\mathbb{Z}_n},*)$ also forms an $\mathbb{F}$-algebra. The {\it Fourier transformation}  $\mathcal{F}: (\mathbb{F}^{\mathbb{Z}_n},*)\rightarrow(\mathbb{F}^{\mathbb{Z}_n},\cdot)$ is defined by
\[(\mathcal{F}(f))(z)=\sum_{i\in\mathbb{Z}_n}f(i)\omega^{iz},~f\in\mathbb{F}^{\mathbb{Z}_n}.\]
It is easy to verify that $\mathcal{F}$ is an algebra isomorphism, and obeys the \textit{inversion formula}
\begin{equation}\label{equ::1}
\mathcal{F}(\mathcal{F}(f))(z)=nf(-z).
\end{equation}
Furthermore, for any divisor $r$ of $n$,   we have 
\[\mathcal{F}\Delta_{r\mathbb{Z}_n}=\frac{n}{r}\Delta_{\frac{n}{r}\mathbb{Z}_n}.\]
In particular, 
\begin{equation}\label{equ::2}
\mathcal{F}1=\mathcal{F}\Delta_{\mathbb{Z}_n}=n\mathcal{F}\Delta_0~\textrm{and }~\mathcal{F}\Delta_0=\Delta_{\mathbb{Z}_n}=1.
\end{equation}

Note that $\mathbb{Z}_n^\ast$ acts on $\mathbb{Z}_n$ by multiplication, and each orbit of this action consists of all elements of a given order in the additive group $\mathbb{Z}_n$. Consequently, each orbit is of  the form  $\mathcal{O}_r=\{c\cdot \frac{n}{r}\in \mathbb{Z}_n\mid c\in \mathbb{Z}_n^\ast\}$, where $r$ is a positive divisor of $n$. The following lemmas present  basic facts about Fourier transformation.
	
\begin{lemma}[{\cite[Corollary 3.2]{MP07}}]\label{lem::2.13}
		If $A$ is a subset of $\mathbb{Z}_{n}$ and $\mathrm{Im}(\mathcal{F}\Delta_A)\subseteq\mathbb{Q}$, then $A$ is a union of some orbits of the action of $\mathbb{Z}_n^\ast$ on $\mathbb{Z}_n$ by multiplication, and $\mathrm{Im}(\mathcal{F}\Delta_A)\subseteq\mathbb{Z}$.
	\end{lemma}

\begin{lemma}[{\cite[Lemma 3.3]{MP07}}]\label{lem::2.14}
 Let $r$ be a  positive divisor  of $n$, and let $\omega$ be a primitive $n$-th root of unity. If $A$ is a subset of $\mathbb{Z}_{n}$, then
$$\mathcal{F}\Delta_{A}\left(\frac{n}{r}\right)=e_0+e_1\xi+\cdots+e_{r-1}\xi^{r-1},$$
where $\xi=\omega^{\frac{n}{r}}$ and $e_i =|A\cap (i+r\mathbb{Z}_n)|$ for $0\leq i\leq r-1$.
\end{lemma}

\section{Proof of the main results}\label{sec::3}
In this section, we shall prove Theorem \ref{thm::main1} and Theorem \ref{thm::main2}. To achieve this goal, we first present a series of lemmas about distance-regular  semi-dihedrants or   pseudo-semi-dihedrants. 
\begin{lemma}\label{lem::3.1}
Let $\SD=\langle\rho,\tau\mid\rho^{2n}=\tau^2=1,\tau\rho\tau=\rho^{n-1} \rangle$ be the semi-dihedral group of order $4n$ with $n=2^r>4$, and let $B$ be a subgroup of $\langle\rho\rangle$ with $|B|>1$. Then $B$ is normal in $\SD$, and $\SD/B$ is isomorphic to $\mathbb{Z}_2$ or a dihedral group. 
\end{lemma}
\begin{proof}
If $\rho^{i}\in B$ for some odd $i$, then $B=\langle\rho\rangle$ because $\gcd(2n,i)=1$. This implies that $B$ is normal in $\SD$ and $\SD/B\cong \mathbb{Z}_2$. Now suppose that $B=\langle\rho^{i}\rangle$ for some even $i\in \mathbb{Z}_{2n}\setminus\{0\}$. Let $i_0$ be the least positive integer such that $\rho^{i_0}\in B$. It is easy to see that  $i_0$ ($i_0\geq 2$) is a factor of $n$, and so $\rho^n\in B=\langle\rho^{i_0}\rangle$. As $\tau \rho^{i_0}\tau=\rho^{(n-1)i_0}=\rho^{-i_0}$, we have $\tau B\tau =B$, and hence $B$ is normal in $\SD$. Also, we see that $\SD/B=\{B,\rho B, \ldots,\rho^{i_0-1}B,\tau B,\rho\tau B, \ldots,\rho^{i_0-1}\tau B\}$. Furthermore, $(\rho B)^{i_0}=(\tau B)^2=B$ and $\tau B\cdot \rho B\cdot \tau B=\tau \rho \tau B=\rho^{n-1} B=\rho^{-1} B=(\rho B)^{-1}$. Therefore, we may conclude that $\SD/B$ is isomorphic to a dihedral group.  
\end{proof}

The case for pseudo-semi-dihedral groups is similar, and we omit the proof.
\begin{lemma}\label{lem::3.2}
Let $\PSD=\langle\rho,\tau\mid\rho^{2n}=\tau^2=1,\tau\rho\tau=\rho^{n+1} \rangle$ be the pseudo-semi-dihedral group of order $4n$ with $n=2^r>4$, and let $B$ be a subgroup of $\langle\rho\rangle$ with $|B|>1$. Then $B$ is normal in $\PSD$, and the quotient group  $\PSD/B\cong\mathbb{Z}_t\oplus\mathbb{Z}_2$ for some $t\mid n$. 
\end{lemma}

Combining Lemmas \ref{lem::2.1}, \ref{lem::2.9} and \ref{lem::3.1}, we get the following result.
\begin{corollary}\label{cor::3.1}
Let $X=\SDG$ be a distance-regular semi-dihedrant with $n=2^r>4$. Then:
\begin{enumerate}[(i)]
    \item if $X$ is antipodal, then the antipodal quotient $\overline{X}$ is a distance-regular circulant or a distance-regular dihedrant;
    \item if $X$ is bipartite, then the halved graph $\frac{1}{2}X$ is a distance-regular circulant, a distance-regular dihedrant, or a distance-regular dicirculant.
\end{enumerate}
\end{corollary}
\begin{proof}
First suppose that $X$ is antipodal. Lemma \ref{lem::2.1} (ii) indicates $\overline{X}$ is distance-regular. Assume that $\mathcal{B}$ is the set of antipodal classes of $X$. Clearly, $\mathcal{B}$ is an imprimitive system for $\SD$ acting on the vertex of $X$ by left multiplication. Let $B\in\mathcal{B}$ be the antipodal class containing the identity $1$. By Lemma \ref{lem::2.9}, $B\le \SD$. If $B\le\langle \rho\rangle$, then from Lemma \ref{lem::3.1} we see that $B$ is normal in $\SD$, and $\SD/B$ is isomorphic to $\mathbb{Z}_2$ or a dihedral group. Moreover, by Lemma \ref{lem::2.9} (i), we have $\overline{X}=\Cay(\SD/B,S/B)$, and hence  $\overline{X}$ is a distance-regular circulant or a distance-regular dihedrant. If $B\cap\langle\rho\rangle\tau\ne \emptyset$, then $\SD=\langle \rho\rangle B$, and Lemma \ref{lem::2.9} (ii) indicates that $\overline{X}$ is isomorphic to a Cayley graph over the group $\langle\rho\rangle/(\langle\rho\rangle\cap B)$, which is cyclic. Therefore, $\overline{X}$ is a distance-regular circulant. This proves (i).

Now suppose that $X$ is bipartite. Then Lemma \ref{lem::2.1} (i) indicates $X^+$ and $X^-$ are both distance-regular. Clearly, $\mathcal{B}=\{V(X^+),V(X^-)\}$ is an imprimitive system for $\SD$, and assume $1\in V(X^+)$. Again by Lemma \ref{lem::2.9},  $V(X^+)$ is a subgroup of $\SD$ with index $2$, which  acts regularly on itself by left multiplication as a subgroup of $\mathrm{Aut}(X^+)$. Note that  each subgroup of $\SD$ with index $2$ is  a cyclic group, a dihedral group or a dicyclic group. Therefore, $X^+$ is a distance-regular Cayley graph over these groups. Since $X$ is vertex transitive, we have $X^-\cong X^+$, and  (ii) follows.
\end{proof}	

Similarly, by using  Lemmas \ref{lem::2.1}, \ref{lem::2.9} and \ref{lem::3.2}, we can deduce the following result for pseudo-semi-dihedral groups. 
\begin{corollary}\label{cor::3.2}
Let $X=\PSDG$ be a distance-regular pseudo-semi-dihedrant with $n=2^r>4$. Then:
\begin{enumerate}[(i)]
    \item if $X$ is antipodal, then the antipodal quotient $\overline{X}$ is a  distance-regular circulant or a distance-regular Cayley graph over  $\mathbb{Z}_t\oplus\mathbb{Z}_2$ with $t\mid n$;
    \item if $X$ is bipartite, then the halved graph $\frac{1}{2}X$ is a  distance-regular circulant or distance-regular Cayley graph over   $2\mathbb{Z}_{2n}\oplus\mathbb{Z}_2\cong \mathbb{Z}_n\oplus\mathbb{Z}_2$.
\end{enumerate}
\end{corollary}

The following two lemmas can be deduced directly from the  definition of Cayley graphs.
\begin{lemma}\label{lem::3.3}
Let $R,T\subseteq\mathbb{Z}_{2n}$ and $0\not\in R$. 
\begin{enumerate}[(i)]
    \item If $X=\SDG$, then $N(\rho^i)=\rho^{i+R}\cup \rho^{i+T}\tau$ and  $N(\rho^i\tau)=\rho^{i+(n-1)T}\cup \rho^{i+(n-1)R}\tau.$
\item If $X=\PSDG$, then $N(\rho^i)=\rho^{i+R}\cup \rho^{i+T}\tau$ and  $N(\rho^i\tau)=\rho^{i+(n+1)T}\cup \rho^{i+(n+1)R}\tau$.
\end{enumerate}
\end{lemma}

\begin{lemma}\label{lem::3.4}
Let  $n=2^r$, $a\in \mathbb{Z}_{2n}^*$ and $b\in2\mathbb{Z}_{2n}$. Then $\SDG\cong \mathrm{SD}(n,aR,b+aT)$ and  $\PSDG\cong\mathrm{PSD}(n,R,n+T)$.
\end{lemma}

Let $X=\Cay(G,S)$ be a Cayley graph of diameter $d$. Recall that 
$$
\mathcal{N}_i=\{g\in G\mid d_X(g,1)=i\},~i=0,1,\ldots,d.
$$
For  $X=\SDG$ or $\PSDG$, we denote  $R_i=\{a\in\mathbb{Z}_{2n}\mid \rho^a\in\mathcal{N}_i\}$ and $T_i=\{b\in\mathbb{Z}_{2n}\mid \rho^b\tau\in\mathcal{N}_i\}$. Clearly, $R_1=R$, $T_1=T$, and both $\cup_{i=0}^d R_i$ and $\cup_{i=0}^d T_i$ are partitions of $\mathbb{Z}_{2n}$.

\begin{lemma}[{\cite[Lemma 4.4, Lemma 4.5]{MP07}}]\label{lem::3.5}
Let $X$ be an antipodal distance-regular dihedrant with diameter $d$. If $d=3$, then $X$ is bipartite, and if  $d=4$, then $X$ is non-bipartite unless $X\cong C_8$.  
\end{lemma}

Analogously, by Lemma \ref{lem::2.6}, we can deduce the following result.
\begin{lemma}\label{lem::3.6}
Let $X$ be an antipodal distance-regular Cayley graph over $\mathbb{Z}_n\oplus \mathbb{Z}_2$ ($s\geq 2$ is even) with diameter $d$. Then $d\leq 3$. Moreover, if $d=3$, then $X$ is bipartite.
\end{lemma}

In what follows, we focus on the characterization of distance-regular semi-dihedrants and    pseudo-semi-dihedrants with additional  restrictions, namely those which are  either antipodal, non-bipartite, and of diameter $3$, or  antipodal, bipartite,  and of diameter $4$. 

\begin{lemma}\label{lem::3.7}
Let $n=2^r>4$. There are no antipodal non-bipartite distance-regular semi-dihedrants or pseudo-semi-dihedrants of order $4n$ with diameter $3$.
\end{lemma}

\begin{proof}
By contradiction, assume that $X=\SDG$ (resp. $X=\PSDG$) is an antipodal non-bipartite distance-regular semi-dihedrant (resp. pseudo-semi-dihedrant) of diameter $3$. Let $k$ and $p$ ($p\geq 2$) denote the valency and the common size of antipodal classes (or fibres) of $X$, respectively. According to Lemma \ref{lem::2.3} (i), $k+1=\frac{4n}{p}$,  and $X$ has the intersection array
$\{k,\mu(p-1),1;1,\mu,k\}$
and eigenvalues $k$, $\theta_1$, $\theta_2=-1$, $\theta_3$, where
\begin{equation}\label{equ::3}
\theta_1=\frac{\lambda-\mu}{2}+\delta,~~\theta_3=\frac{\lambda-\mu}{2}-\delta~~\mbox{and}~~\delta=\sqrt{k+\left(\frac{\lambda-\mu}{2}\right)^2}.
\end{equation}
If $k=2$, then $X\cong C_{4n}$, which is impossible because $X$ is non-bipartite. If $k=3$, then $k+1=\frac{4n}{p}$ indicates $p=n$, and so $k=\mu(p-1)+\lambda+1=\mu(n-1)+\lambda+1\geq 3\mu+1\geq 4$,  a contradiction. Therefore, $k\geq 4$ and $p<n$. Let $H=\mathcal{N}_3\cup\{1\}$. Then $H$ is an antipodal class of $X$, and $|H|=p$. Since $\SD$ (resp. $\PSD$) acts regularly on  $V(X)$ by left multiplication, the antipodal classes of $X$ form an imprimtivity system for $\SD$ (resp. $\PSD$). By Lemma \ref{lem::2.9}, $H$ is a subgroup of $\SD$ (resp. $\PSD$). If $p$ is not prime, then $H$ has a non-trivial subgroup $K$ contained in $\langle\rho\rangle$. Let $\mathcal{B}$ denote the partition consisting of all orbits of $K$ acting on $V(X)$ by left multiplication. As $K$ is normal in $\SD$ (resp. $\PSD$) by Lemma \ref{lem::3.1} (resp. Lemma \ref{lem::3.2}), the partition $\mathcal{B}$ is also an imprimtivity system for $\SD$, and it follows from  Lemma \ref{lem::2.9} (i) and Lemma \ref{lem::3.1} (resp. Lemma \ref{lem::3.2}) that the quotient graph $X_\mathcal{B}$ is a dihedrant (resp. Cayley graph over $\mathbb{Z}_t\oplus\mathbb{Z}_2$  with $t\mid n$). Observe that $\mathcal{B}$ is an equitable partition of $X$, and each block of $\mathcal{B}$ is contained in some fibre of $X$ and is neither a single vertex nor a fibre. By Lemma \ref{lem::2.2}, $X_\mathcal{B}$ is an antipodal distance-regular graph with diameter $3$. If $X_\mathcal{B}$ is bipartite, then $X$ is also bipartite, a contradiction. Hence, $X_\mathcal{B}$ is an antipodal non-bipartite distance-regular dihedrant (resp. Cayley graph over $\mathbb{Z}_t\oplus\mathbb{Z}_2$ with $t\mid n$) of diameter $3$, which is impossible by Lemma \ref{lem::3.5} (resp. Lemma \ref{lem::3.6}). Therefore, $p$ is a prime number, and thus $p=2$ due to $p\mid 4n$ and $n=2^r$. This implies that  $\mathcal{N}_3=\{\rho^c\tau\}$ for some $c\in 2\mathbb{Z}_{2n}$ (resp. $c\in\{0,n\}$) or $\mathcal{N}_3=\{\rho^n\}$. Also note that  $X$ has the intersection array $\{k,\mu,1;1,\mu,k\}$, where $k=2n-1=\mu+\lambda+1$. We consider the following two situations.

{\flushleft\bf Case 1.} $\mathcal{N}_3=\{\rho^c\tau\}$ for some $c\in 2\mathbb{Z}_{2n}$ (resp. $c\in\{0,n\}$). 

By Lemma \ref{lem::3.4}, we may assume that $\mathcal{N}_3=\{\tau\}$ in this situation. 
\begin{claim}\label{claim::1}
The following two statements hold.
\begin{enumerate}[(i)]
    \item If $X=\SDG$, then $R_2=T=-T=(n-1)T$, $T_2=R=-R=(n-1)R$, and $\underline{\mathbf{r}}+\underline{\mathbf{t}}=2n\Delta_0-1$. 
    \item If $X=\PSDG$, then $R_2=T=-T=(n+1)T$, $T_2=R=-R=(n+1)R$, and $\underline{\mathbf{r}}+\underline{\mathbf{t}}=2n\Delta_0-1$.
\end{enumerate}
\end{claim}
\begin{proof}
We only prove the claim for $X=SD(n,R,T)$ since the case $X=PSD(n,R,T)$ is similar. Note that $T=(n+1)T=-(n-1)T$. By Lemma \ref{lem::3.3}, $N(\tau)=\rho^{(n-1)T}\cup \rho^{(n-1)R}\tau=\rho^{-T}\cup \rho^{(n-1)R}\tau$. On the other hand, $N(\tau)=\mathcal{N}_2=\rho^{R_2}\cup \rho^{T_2} \tau$. Thus $R_2=-T$ and $T_2=(n-1)R$. Since both $R\cup R_2$ and $T\cup T_2$ are partitions of $\mathbb{Z}_{2n}\setminus\{0\}$ and $R=-R$, we have $-T=R_2=-R_2=T$, and so $R=T_2=(n-1)R$. Moreover, we conclude that $\underline{\mathbf{r}}+\underline{\mathbf{t}}=2n\Delta_0-1$ because $R\cup R_2=R\cup T$ is a partition of $\mathbb{Z}_{2n}\setminus\{0\}$. 
\end{proof}

\begin{claim}\label{claim::2}
\rm $\underline{\mathbf{r}}^2+\underline{\mathbf{t}}^2=k+\lambda \underline{\mathbf{r}}+\mu \underline{\mathbf{t}}$ and $2 \underline{\mathbf{rt}}=\lambda \underline{\mathbf{t}}+\mu \underline{\mathbf{r}}
$.
\end{claim}

\begin{proof}
We only prove the claim for $X=\SDG$. Note that 
\[\begin{array}{ll}
     &(\Delta_R*\Delta_R)(i)+(\Delta_T*\Delta_{-T})(i)=|R\cap (i-R)|+|T\cap (i+T)|  \\[2mm]
     =&(k\Delta_0+\lambda\Delta_R+\mu\Delta_{R_2})(i)=(k\Delta_0+\lambda\Delta_R+\mu\Delta_{T})(i), 
\end{array}\] and
\[\begin{array}{ll}
     &2(\Delta_R*\Delta_T)(i)=|R\cap (i-T)|+|T\cap (i-R)|=|R\cap (i+(n-1)T)|+|T\cap (i+(n-1)R)|  \\[2mm]
     =&(\lambda \Delta_T+\mu \Delta_{T_2})(i)=(\lambda \Delta_T+\mu \Delta_{R})(i). 
\end{array}\] As $T=-T$, we have $\underline{\mathbf{r}}^2+\underline{\mathbf{t}}^2=\underline{\mathbf{r}}^2+|\underline{\mathbf{t}}|^2=k+\lambda \underline{\mathbf{r}}+\mu \underline{\mathbf{t}}$ and $2 \underline{\mathbf{rt}}=\lambda \underline{\mathbf{t}}+\mu \underline{\mathbf{r}}
$.
\end{proof}

Let $\underline{\mathbf{x}}=\underline{\mathbf{r}}-\underline{\mathbf{t}}$.  According to Claim \ref{claim::2}, we have $\underline{\mathbf{x}}^2-(\lambda-\mu)\underline{\mathbf{x}}-k=0$, implying that $\underline{\mathbf{x}}(z)\in\{\theta_1,\theta_3\}$  for all $z\in \mathbb{Z}_{2n}$, where $\theta_1$ and $\theta_3$ are given in \eqref{equ::3}. In particular, $\underline{\mathbf{x}}(0)=|R|-|T|\in\{\theta_1,\theta_3\}$, and hence $\theta_1$ and $\theta_3$ are integers. Moreover, if $|R|>|T|$ then $\underline{\mathbf{x}}(0)=\theta_1$, and if $|R|<|T|$ then $\underline{\mathbf{x}}(0)=\theta_3$. 
By Claim \ref{claim::1}, we have  $X_2=\mathrm{SD}(n,R_2,T_2)=\mathrm{SD}(n,T,R)$ (resp. $X_2=\mathrm{PSD}(n,R_2,T_2)=\mathrm{PSD}(n,T,R)$), which is also an antipodal non-bipartite distance-regular graph with the same parameters as $X$. Therefore, we may always assume $|R|>|T|$ and $|R|-|T|=\theta_1$. Let $B=\{z\in\mathbb{Z}_{2n}\mid \underline{\mathbf{x}}(z)=\theta_1\}$ and $C=\{z\in\mathbb{Z}_{2n}\mid \underline{\mathbf{x}}(z)=\theta_3\}$. By Claim \ref{claim::1},  $\underline{\mathbf{t}}=2n\Delta_0-1-\underline{\mathbf{r}}$. Thus we have
\[\underline{\mathbf{r}}(z)=\left\{\begin{array}{ll}
|R|, & z=0, \\
\frac{\theta_1-1}{2}, & z \in B, \\
\frac{\theta_3-1}{2}, & z \in C,
\end{array} 
~\textrm{and}~\underline{\mathbf{t}}(z)= \begin{cases}|T|, & z=0, \\
\frac{-\theta_1-1}{2}, & z \in B, \\
\frac{-\theta_3-1}{2}, & z \in C,\end{cases}\right.
\]
or equivalently, 
\begin{equation}\label{equ::1}
\underline{\mathbf{r}}=|R|\Delta_0+\frac{\theta_1-1}{2}\Delta_B+\frac{\theta_3-1}{2}\Delta_C\textrm{ and }    \underline{\mathbf{t}}=|T|\Delta_0+\frac{-\theta_1-1}{2}\Delta_B+\frac{-\theta_3-1}{2}\Delta_C.
\end{equation}
Applying the Fourier transformation on both sides of the equations in \eqref{equ::1}, and using \eqref{equ::1} and \eqref{equ::2}, we can deduce that
\begin{equation}\label{equ::5}
    \left\{\begin{array}{l}
     \mathcal{F}(\Delta_B)=|B|\Delta_0+(\frac{n}{\delta}-1)\Delta_R-(\frac{n}{\delta}+1)\Delta_T,  \\[2mm]
     \mathcal{F}(\Delta_C)=|C|\Delta_0-\frac{n}{\delta} \Delta_R+\frac{n}{\delta} \Delta_T,
\end{array}\right.
\end{equation}
where $\delta$ is given in \eqref{equ::3}. Since $\textrm{Im}(\mathcal{F}(\Delta_B)),  \textrm{Im}(\mathcal{F}(\Delta_C))\in \mathbb{Q}$, Lemma \ref{lem::2.13} indicates that both $B$ and $C$ are the union of some orbits of $\mathbb{Z}_{2n}^*$ acting on $\mathbb{Z}_{2n}$ by multiplication, and hence $B=-B$ and $C=-C$. Furthermore, for any $i\in\mathbb{Z}_{2n}$, we have
\[\begin{aligned}
     |C\cap (i-C)|&=(\Delta_C*\Delta_C)(i)\\
     &=\frac{1}{2n}\mathcal{F}\left(\mathcal{F}(\Delta_C*\Delta_C)\right)(-i)\\
     &=\frac{1}{2n}\mathcal{F}\left(\mathcal{F}(\Delta_C)^2\right)(-i)\\
     &=\frac{1}{2n}\mathcal{F}\left(|C|^2\Delta_0+\frac{n^2}{\delta^2}(\Delta_R+\Delta_T)\right)(-i)\\
     &=\frac{1}{2n}\mathcal{F}\left(|C|^2\Delta_0+\frac{n^2}{\delta^2}\Delta_{\mathbb{Z}_{2n}\setminus\{0\}}\right)(-i)\\
     &=\frac{1}{2n}(|C|^2-\frac{n^2}{\delta^2})+\frac{n^2}{\delta^2}\Delta_0(i).
\end{aligned}
\]
This implies that $\Cay(\mathbb{Z}_{2n},C)$ is a strongly regular circulant graph with parameters $(2n,|C|,\lambda',\mu')$ where $\lambda'=\mu'=\frac{1}{2n}(|C|^2-\frac{n^2}{\delta^2})$, or is isomorphic to $K_{2n}$ (if $C=\mathbb{Z}_{2n}\setminus\{0\}$) or $nK_2$ (if $C=\{n\}$). By Theorem \ref{thm::cir}, there are no strongly regular circulants with $\lambda'=\mu'$. If $\Cay(\mathbb{Z}_{2n},C)\cong K_{2n}$, then $C=\mathbb{Z}_{2n}\setminus\{0\}$, and $B=\emptyset$, contrary to \eqref{equ::5}. If $\Cay(\mathbb{Z}_{2n},C)\cong nK_{2}$, then $C=\{n\}$, and thus $\mathcal{F}(\Delta_C)=2\Delta_{2\mathbb{Z}_{2n}}-1$. Combining this with \eqref{equ::5}, we can deduce that   $R=2\mathbb{Z}_{2n}+1$ and $T=2\mathbb{Z}_{2n}$. Therefore, $X$ is bipartite, contrary to our assumption.

{\flushleft\bf Case 2.} $\mathcal{N}_3=\{\rho^n\}$.

In this case, we have 
\[N(\rho^n)=\rho^n(\rho^R\cup \rho^T\tau)=\rho^{n+R}\cup \rho^{n+T}\tau.\]
Therefore, $R_2=n+R$. Note that $\frac{n}{2}\in R\cup R_2=\mathbb{Z}_{2n}\setminus\{0,n\}$. If $\frac{n}{2}\in R$, then $-\frac{n}{2}=n+\frac{n}{2}\in R_2$, contrary to   $-\frac{n}{2}\in -R=R$. If $\frac{n}{2}\in R_2$, then $-\frac{n}{2}=n+\frac{n}{2}\in R$, and hence $\frac{n}{2}\in -R=R$, a contradiction.

The proof is completed.
\end{proof}

\begin{lemma}\label{lem::3.8}
Let $n=2^r>4$, and let $X=\SDG$ be a semi-dihedrant of order $4n$ with valency at least $3$. Then $X$ is an antipodal bipartite distance-regular graph  with diameter $4$ if and only if $R=-R\subseteq 2\mathbb{Z}_{2n}+1$, $T\subseteq 2\mathbb{Z}_{2n}$, $R\cap (n+R)=T\cap (n+T)=\emptyset$, $|R|=|T|=n/2$, and $|R\cap (i+R)|+|T\cap (i+T)|=n/2$ for all $i\in 2\mathbb{Z}_{2n}\setminus\{0,n\}$.
\end{lemma}
\begin{proof}
First we consider the necessity. Assume that $X=\SDG$ is an antipodal bipartite distance-regular semi-dihedrant of diameter $d=4$ and valency $k\geq 3$. Let $p$ ($p\geq 2$) denote the common size of antipodal classes (or fibres) of $X$. By Lemma \ref{lem::2.3} (ii), we have $4n=2p^2\mu$, $k=p\mu$, and $X$ has the intersection array $\{p\mu, p\mu-1,(p-1)\mu, 1;1,\mu,p\mu-1,p\mu\}$. As $p$ is a divisor of $4n=2^{r+2}$, $k$ is even. Thus $k\geq 4$ and  $p\leq n/2$. Similarly as in the proof of Lemma \ref{lem::3.7}, we can deduce that $p=2$. Thus $k=n$, $\mu=n/2$, and  $\mathcal{N}_4=\{\rho^c\tau\}$ for some $c\in 2\mathbb{Z}_{2n}$ or $\mathcal{N}_4=\{\rho^n\}$.  Furthermore, since $\mathcal{N}_0\cup \mathcal{N}_2\cup \mathcal{N}_4$ is a subgroup of index $2$ in $\SD$, we have $\mathcal{N}_0\cup \mathcal{N}_2\cup \mathcal{N}_4=\langle \rho\rangle$, $\langle \rho^2,\tau\rangle$, or $\langle \rho^2,\rho\tau\rangle$. 
Then it suffices to consider the following four situations. 

{\flushleft\bf Case 1.} $\mathcal{N}_0\cup \mathcal{N}_2\cup \mathcal{N}_4=\langle \rho^2, \tau\rangle$, and $\mathcal{N}_4=\{\rho^c\tau\}$ for some $c\in 2\mathbb{Z}_{2n}$. 

By Lemma \ref{lem::3.4}, we may assume that $\mathcal{N}_4=\{\tau\}$. Then $R_2=T_2=2\mathbb{Z}_{2n}\setminus\{0\}$, and so $\underline{\mathbf{r}}_2=\underline{\mathbf{t}}_2=n\Delta_{n\mathbb{Z}_{2n}}-1$. Moreover, from $\rho^{R_3}\cup \rho^{T_3}\tau=\mathcal{N}_3=N(\tau)=\rho^{(n-1)T}\cup \rho^{(n-1)R}\tau=\rho^{-T}\cup \rho^{(n-1)R}\tau$ we can deduce that $R_3=-T$. Since both $R\cup R_3$ and $T\cup T_3$ are partitions of $2\mathbb{Z}_{2n}+1$ and $R=-R$, we have $-T=R_3=-R_3=T$, and thus $T_3=R$.   Therefore, $\mathrm{Im}(\underline{\mathbf{t}})\subseteq\mathbb{R}$. Then, as in Case 1 of Lemma \ref{lem::3.7}, we have
$$
\underline{\mathbf{r}}^2+\underline{\mathbf{t}}^2=k+\mu \underline{\mathbf{t}_2}=\frac{n^2}{2}\Delta_{n\mathbb{Z}_{2n}}+\frac{n}{2}~\mbox{and}~2 \underline{\mathbf{rt}}=\mu \underline{\mathbf{t}_2}=\frac{n^2}{2}\Delta_{n\mathbb{Z}_{2n}}-\frac{n}{2}.
$$
This implies that $\underline{\mathbf{r}}(n/2)\in\{\sqrt{n}/2,-\sqrt{n}/2\}$. However, since $R\subseteq 2\mathbb{Z}_{2n}+1$, by Lemma \ref{lem::2.14}, the real part of $\underline{\mathbf{r}}(n/2)$ is $0$, which is a contradiction.

{\flushleft\bf Case 2.} $\mathcal{N}_0\cup \mathcal{N}_2\cup \mathcal{N}_4=\langle \rho^2, \tau\rangle$, and $\mathcal{N}_4=\{\rho^n\}$. 

In this situation, we have $R_2=2\mathbb{Z}_{2n}\setminus \{0,n\}$ and $T_2=2\mathbb{Z}_{2n}$, and hence $R,T\subseteq 2\mathbb{Z}_{2n}+1$. Take $2l+1\in T$. As  $\rho^{2l+1}\tau\cdot \rho^{2l+1}\tau=\rho^n$, we assert that $n\in R_2$, which is a contradiction.

{\flushleft\bf Case 3.} $\mathcal{N}_0\cup \mathcal{N}_2\cup \mathcal{N}_4=\langle \rho\rangle$, and $\mathcal{N}_4=\{\rho^n\}$. 

In this situation, $R=\emptyset$, and Lemma \ref{lem::2.11} indicates that $\tau\rho^T\tau=\rho^{(n-1)T}=\rho^{-T}$ is a symmetric $(n,2,n,\frac{n}{2})$-relative difference set relative to $\langle\rho^n\rangle$ in $\langle\rho\rangle$. However, by  Lemma \ref{lem::2.12}, this is impossible.

{\flushleft\bf Case 4.} $\mathcal{N}_0\cup \mathcal{N}_2\cup \mathcal{N}_4=\langle \rho^2, \rho\tau\rangle$, and $\mathcal{N}_4=\{\rho^n\}$. 

In this situation, $R_2=2\mathbb{Z}_{2n}\setminus\{0,n\}$ and $T_2=2\mathbb{Z}_{2n}+1$. Then $R\cup R_3$ and $T\cup T_3$ are partitions of  $2\mathbb{Z}_{2n}+1$ and $2\mathbb{Z}_{2n}$, respectively. Furthermore, from $\rho^{R_3}\cup \rho^{T_3}\tau=\mathcal{N}_3=N(\rho^n)=\rho^{n+R}\cup \rho^{n+T}\tau$ we can deduce that $R_3=n+R$ and $T_3=n+T$. Therefore, $|R|=|R_3|=|T|=|T_3|=\frac{n}{2}$, and $R\cap (n+R)=T\cap (n+T)=\emptyset$. For any $i\in R_2=2\mathbb{Z}_{2n}\setminus\{0,n\}$, since $\rho^i\in \mathcal{N}_2$ and $X$ is a distance-regular graph with intersection array $\{n,n-1,n/2,1; 1, n/2, n-1,n\}$, we have 
$|N(1)\cap N(\rho^i)|=|(\rho^R\cup \rho^T\tau)\cap (\rho^{i+R}\cup \rho^{i+T}\tau)|=|R\cap (i+R)|+|T\cap (i+T)|=\mu=n/2$, as required.

Now we consider the sufficiency. Assume that  $R\subseteq 2\mathbb{Z}_{2n}+1$, $T\subseteq 2\mathbb{Z}_{2n}$, $R\cap (n+R)=T\cap (n+T)=\emptyset$, $|R|=|T|=n/2$, and $|R\cap (i+R)|+|T\cap (i+T)|=n/2$ for all $i\in 2\mathbb{Z}_{2n}\setminus\{0,n\}$. Then $|N(1)\cap N(\rho^i)|=n/2$ for any $i\in 2\mathbb{Z}_{2n}\setminus\{0,n\}$ according to the above arguments, and it follows that $2\mathbb{Z}_{2n}\setminus\{0,n\}\subseteq R_2$. On the other hand, by the assumption, we see that  $R_2\subseteq 2\mathbb{Z}_{2n}$. Since    $|N(1)\cap N(\rho^n)|=|R\cap (n+R)|+|T\cap (n+T)|=0$, we have   $n\not\in R_2$, and hence $R_2=2\mathbb{Z}_{2n}\setminus\{0,n\}$. Moreover, for any $i\in 2\mathbb{Z}_{2n}+1$, we have $|N(1)\cap N(\rho^i\tau)|=|(\rho^R\cup \rho^T\tau)\cap (\rho^{i+(n-1)R}\tau\cup \rho^{i+(n-1)T)})|=|R\cap (i+(n-1)T)|+|T\cap (i+(n-1)R)|=|R\cap (i-T)|+|T\cap (i+n-R)|=|T\cap (i-R)|+|T\cap (i+n-R)|=|T\cap (i+R)|+|T\cap (i+n+R)|=|T\cap ((i+R)\cup (i+n+R))|=|T\cap (i+2\mathbb{Z}_{2n}+1)|=|T\cap 2\mathbb{Z}_{2n}|=|T|=n/2$, and it follows that $2\mathbb{Z}_{2n}+1\subseteq T_2$. Also, by the assumption, we can deduce that $T_2\subseteq 2\mathbb{Z}_{2n}+1$, and hence  $T_2=2\mathbb{Z}_{2n}+1$. Furthermore, since $R\subseteq 2\mathbb{Z}_{2n}+1$, 
$T\subseteq 2\mathbb{Z}_{2n}$, $R_2=2\mathbb{Z}_{2n}\setminus\{0,n\}$ and $T_2=2\mathbb{Z}_{2n}+1$, we have  $R_3\subseteq 2\mathbb{Z}_{2n}+1$ and $T_3\subseteq 2\mathbb{Z}_{2n}$. Then  $\rho^n\in \mathcal{N}_4$ because $n/2\in R_2$ and  $\rho^n=\rho^{n/2}\cdot \rho^{n/2}$. Note that $N(\rho^n)=\rho^{n+R}\cup \rho^{n+T}\tau$. For any $i\in R$,  there exists some $j\in R$ such that $j-i\not\in\{0,n\}$ because $|R|=n/2>2$ and $R\cap (n+R)=\emptyset$. Take $j_i=n-(j-i)$, we have $j_i\in R_2=2\mathbb{Z}_{2n}\setminus\{0,n\}$. Then $\rho^{j_i}\cdot \rho^j=\rho^{n+i}$ indicates that $\rho^{n+i}\in N(\rho^{j_i})$, and hence $\rho^{n+i}\in \mathcal{N}_3$, i.e., $n+i\in R_3$. By the arbitrariness of $i\in R$, we obtain $n+R\subseteq R_3$. Since $|R|=n/2$, $R\cap (n+R)=\emptyset$ and $R\subseteq 2\mathbb{Z}_{2n}+1$, we have $R\cup (n+R)=2\mathbb{Z}_{2n}+1$.  Combining this with $n+R\subseteq R_3\subseteq 2\mathbb{Z}_{2n}+1$ and $R\cap R_3=\emptyset$, we assert that $R_3=n+R$. Similarly, we can deduce that $T_3=n+T$. Thus we may conclude that $\mathcal{N}_3=\rho^{n+R}\cup \rho^{n+T}\tau$, $\mathcal{N}_4=\{\rho^n\}$,  $\mathcal{N}_0\cup \mathcal{N}_2 \cup \mathcal{N}_4=\langle\rho^2,\rho\tau\rangle$, and $X=\SDG$ is  a bipartite graph with  diameter $4$. Furthermore, by above arguments,  we have $|N(1)\cap N(\rho^i)|=|N(1)\cap N(\rho^j\tau)|=n/2$ for any $i\in R_2=2\mathbb{Z}_{2n}\setminus\{0,n\}$ and  $j\in T_2=2\mathbb{Z}_{2n}+1$. Combining this with the fact that $X$ is bipartite and vertex-transitive, it is routine to verify that  $X=\SDG$ is  distance-regular  with intersection array $\{n,n-1,n/2,1; 1, n/2, n-1,n\}$.

This completes the proof.
\end{proof}

\begin{remark}
\rm For $n=8$ (resp. $n=32$), by using computer search, we find that $R=\{5,7, 9,11\}$ and $T=\{4,8,10,14\}$ (resp. $R=\{9, 11, 15, 19, 25, 27,29, 31, 33, 35, 37, 39, 45, 49, 53, 55\}$ and $T=\{8, 12, 14, 22, 24, 30, 34, 36, 38, 42, 48, 50, 52, 58, 60, 64\}$) are a pair of sets satisfying the sufficient condition in Lemma \ref{lem::3.8}. This implies the existence of an antipodal bipartite distance-regular semi-dihedrant of order $32$ (resp. $128$) with diameter $4$. However, for $n=16$, there are no such pairs.
\end{remark}

\begin{remark}
\rm According to Lemma \ref{lem::2.11}, it is easy to see that the sufficient condition in Lemma \ref{lem::3.8} is exactly equivalent to the condition that 
$\rho^{-1}(\rho^R \cup \rho^{T}\tau)=\rho^{-1+R} \cup \rho^{-1+T}\tau$ is a symmetric $(n,2,n,n/2)$-relative difference set relative to $\langle\rho^n\rangle$ in the group $\langle\rho^2,\rho\tau\rangle$. 
\end{remark}

The following result is an analogue of Lemma \ref{lem::3.8} for pseudo-semi-dihedrant graphs, and the proof is a little different.

\begin{lemma}\label{lem::3.9}
Let $n=2^r>2$, and let $X=\PSDG$ be a pseudo-semi-dihedrant of order $4n$ with valency at least $3$. Then $X$ is an antipodal bipartite distance-regular graph with diameter $4$ if and only if $R=-R$ and $T=(n-1)T=n-T$ are subsets of $2\mathbb{Z}_{2n}+1$ such that $R\cap (n+R)=T\cap (n+T)=\emptyset$, $|R|=|T|=n/2$, and $|R\cap (i+R)|+|T\cap (i+T)|=n/2$ for all $i\in 2\mathbb{Z}_{2n}\setminus\{0,n\}$.
\end{lemma}
\begin{proof}
Assume that $X=\PSDG$ is an antipodal bipartite distance-regular semi-dihedrant of diameter $d=4$ and valency $k\geq 3$. Let $p$ ($p\geq 2$) denote the common size of antipodal classes of $X$. By a similar discussion as in the proof of Lemma \ref{lem::3.8}, we obtain $p=2$, $k=n$, $\mu=n/2$, and  $\mathcal{N}_4=\{\rho^c\tau\}$ for $c\in\{0,n\}$ or $\mathcal{N}_4=\{\rho^n\}$.  Furthermore, since $\mathcal{N}_0\cup \mathcal{N}_2\cup \mathcal{N}_4$ is a subgroup of index $2$ in $\PSD$, we have $\mathcal{N}_0\cup \mathcal{N}_2\cup \mathcal{N}_4=\langle \rho\rangle$, $\langle \rho^2,\tau\rangle$ or $\langle \rho^2,\rho\tau\rangle$. It suffices to consider the following four situations. 

{\flushleft\bf Case 1.} $\mathcal{N}_0\cup \mathcal{N}_2\cup \mathcal{N}_4=\langle \rho^2,\tau\rangle$, and $\mathcal{N}_4=\{\rho^c\tau\}$ for $c\in\{0,n\}$. 

In this situation, by using the same method as in Case 1 of the proof of Lemma \ref{lem::3.8}, we can deduce a contradiction.

{\flushleft\bf Case 2.} $\mathcal{N}_0\cup \mathcal{N}_2\cup \mathcal{N}_4=\langle \rho\rangle$, and $\mathcal{N}_4=\{\rho^n\}$.

In this situation, $R=\emptyset$, and Lemma \ref{lem::2.11} indicates that $\tau\rho^T\tau=\rho^{(n+1)T}=\rho^{-T}$ is a symmetric $(n,2,n,\frac{n}{2})$-relative difference set relative to $\langle\rho^n\rangle$ in $\langle\rho\rangle$, which is impossible by Lemma \ref{lem::2.12}.

{\flushleft\bf Case 3.} $\mathcal{N}_0\cup \mathcal{N}_2\cup \mathcal{N}_4=\langle \rho^2,\rho\tau\rangle$, and $\mathcal{N}_4=\{\rho^n\}$.

Note that $\rho^{n+2}= (\rho\tau)^2\in\langle \rho\tau\rangle$ and  $\rho^4=(\rho\tau)^4\in\langle\rho\tau \rangle$. As $4\mid n$, we have $\rho^n\in\langle\rho\tau\rangle$, and hence  $\rho^2\in\langle\rho\tau\rangle$. Therefore, $\langle\rho^2,\rho\tau\rangle=\langle\rho\tau\rangle$ is a cyclic group of order $2n$. By  Lemma \ref{lem::2.11},  $\rho^{-1}(\rho^R \cup \rho^{T}\tau)=\rho^{-1+R} \cup \rho^{-1+T}\tau$ is a symmetric $(n,2,n,n/2)$-relative difference set relative to $\langle \rho^n \rangle $ in the cyclic group $\langle \rho\tau\rangle$, which contradicts Lemma \ref{lem::2.12}.

{\flushleft\bf Case 4.} 
 $\mathcal{N}_0\cup \mathcal{N}_2\cup \mathcal{N}_4=\langle \rho^2, \tau\rangle$, and $\mathcal{N}_4=\{\rho^n\}$. 
 
 In this situation, we have $R_2=2\mathbb{Z}_{2n}\setminus\{0,n\}$, $T_2=2\mathbb{Z}_{2n}$, and $R,T\subseteq 2\mathbb{Z}_{2n}+1$. Note that $R=-R$ and $T=(n-1)T=n-T$. Then both $R\cup R_3$ and $T\cup T_3$ are partitions of  $2\mathbb{Z}_{2n}+1$. Furthermore, from $\rho^{R_3}\cup \rho^{T_3}\tau=\mathcal{N}_3=N(\rho^n)=\rho^{n+R}\cup \rho^{n+T}\tau$ we can deduce that $R_3=n+R$ and $T_3=n+T$. Therefore, $|R|=|R_3|=|T|=|T_3|=\frac{n}{2}$, and $R\cap (n+R)=T\cap (n+T)=\emptyset$. For any $i\in R_2=2\mathbb{Z}_{2n}\setminus\{0,n\}$, since $\rho^i\in \mathcal{N}_2$ and $X$ is a distance-regular graph with intersection array $\{n,n-1,n/2,1; 1, n/2, n-1,n\}$, we have 
$|N(1)\cap N(\rho^i)|=|(\rho^R\cup \rho^T\tau)\cap (\rho^{i+R}\cup \rho^{i+T}\tau)|=|R\cap (i+R)|+|T\cap (i+T)|=\mu=n/2$, as required.

Conversely, assume that $R=-R$ and $T=(n-1)T=n-T$ are subsets of $2\mathbb{Z}_{2n}+1$ such that $R\cap (n+R)=T\cap (n+T)=\emptyset$, $|R|=|T|=n/2$, and $|R\cap (i+R)|+|T\cap (i+T)|=n/2$ for all $i\in 2\mathbb{Z}_{2n}\setminus\{0,n\}$. By using the same method as in  Case 4 of the proof of Lemma \ref{lem::3.8}, we can deduce that $X=\PSDG$ is an antipodal bipartite distance-regular graph with intersection array $\{n,n-1,n/2,1; 1, n/2, n-1,n\}$.

This completes the proof.
\end{proof}

\begin{remark}
\rm For $n=8$, by using computer search, we find that $R=\{5,7,9,11\}$ and $T=\{3,5,9,15\}$
are exactly a pair of sets satisfying the sufficient condition in Lemma \ref{lem::3.9}. This implies the existence of an antipodal bipartite distance-regular pseudo-semi-dihedrant of order $32$ with diameter $4$. However, for $n\in\{16, 32,64\}$, there are no such pairs.
\end{remark}

\begin{remark}
\rm According to Lemma \ref{lem::2.11}, it is easy to see that the sufficient condition in Lemma \ref{lem::3.9} is exactly equivalent to the condition that 
$\rho^{-1}(\rho^R \cup \rho^{T}\tau)=\rho^{-1+R} \cup \rho^{-1+T}\tau$ is a symmetric $(n,2,n,n/2)$-relative difference set relative to $\langle\rho^n\rangle$ in the group $\langle\rho^2,\tau\rangle$. Note that $\langle\rho^2,\tau\rangle=\langle\rho^2\rangle\times\langle\tau\rangle \cong 2\mathbb{Z}_{2n}\oplus\mathbb{Z}_2$. Thus the condition is also equivalent to the condition that $(-1+R,0)\cup(-1+T,1)$ is a symmetric $(n,2,n,n/2)$-relative difference set relative to $\langle (n,0) \rangle $ in $2\mathbb{Z}_{2n}\oplus\mathbb{Z}_2$. 
\end{remark}

Now we are in a position to give the proof of Theorem \ref{thm::main1} and Theorem \ref{thm::main2}.

\renewcommand\proofname{\bf{Proof of Theorem \ref{thm::main1}}}
\begin{proof}
First of all, we note that the graphs in  (i)--(iii) are trivial distance-regular semi-dihedrants because $K_{4n}\cong \Cay(\SD,\SD\setminus\{1\})$,   $K_{s\times t}\cong \Cay(\SD, \SD\setminus H)$ with $H$ being a subgroup of $\SD$ of order $t$ where $t\mid 4n$, and  $K_{2n,2n}-2nK_2\cong \Cay(\SD,\rho^{\mathbb{Z}_{2n}\setminus\{0\}}\tau)$. Furthermore, by Lemma \ref{lem::2.10}, the graphs in (iv)--(vi) are non-antipodal bipartite non-trivial distance-regular semi-dihedrants with diameter $3$. Also, by Lemma \ref{lem::3.8},  the graph in (vii) is a $2$-fold antipodal bipartite non-trivial distance-regular semi-dihedrant with diameter $4$ (i.e., Hadamard graph).

Now suppose that  $X=\SDG$ is a distance-regular semi-dihedrant of diameter $d$. Clearly, $X$ cannot be a cycle because $n>4$. If $X$ is trivial, then $X$ would be one of the graphs listed in (i)--(iii). If $X$ is non-trivial, then $k\geq 3$, and Corollary \ref{cor::2.1} indicates that  $X$ is imprimitive, and it suffices to consider the following three cases.

{\flushleft \bf Case A.} $X$ is antipodal but not bipartite.

By Lemma \ref{lem::2.1} and Corollary \ref{cor::3.1}, the antipodal quotient $\overline{X}$ of $X$ is a primitive distance-regular circulant or dihedrant. Then it follows from Theorem \ref{thm::cir} and Theorem  \ref{thm::dih}  that $\overline{X}$ is a complete graph, a cycle of prime order, or a Paley graph of prime order.  If $\overline{X}$ is a cycle of prime order, then $X$ would be a cycle, which is impossible because $X$ is non-trivial. If $\overline{X}$ is a Paley graph of prime order, by  Lemma \ref{lem::2.4}, we also deduce a contradiction. Thus  $\overline{X}$ is a complete graph, and so $d=2$ or $3$ according to Lemma \ref{lem::2.1}. By Lemma \ref{lem::3.7}, $d\neq 3$, whence $d=2$. However, complete multipartite graphs are the only antipodal distance-regular graphs with diameter $2$. Therefore, there are no non-trivial distance-regular semi-dihedrants which are antipodal but not bipartite.

{\flushleft \bf Case B.} $X$ is antipodal and bipartite.

If $d$ is odd, by Lemma \ref{lem::2.1}, $\overline{X}$ is primitive. Also, by Corollary \ref{cor::3.1},   $\overline{X}$ is a distance-regular circulant or dihedrant. 
As in Case A, we assert that $\overline{X}$ is a complete graph. Hence, $d=3$. Considering that $X$ is antipodal and bipartite, we obtain  $X\cong K_{2m,2m}-2mK_2$, which is impossible because $X$ is non-trivial.
Now suppose that $d$ is even. Then Corollary \ref{cor::3.1} and Lemma \ref{lem::2.1} indicate that $\frac{1}{2}X$ is an antipodal non-bipartite distance-regular circulant, dihedrant, or dicirculant with diameter $d_{\frac{1}{2}X}=d/2$.  Clear, $d\neq 2$. If $d=4$, by Lemma \ref{lem::3.8}, we conclude that 
$R=-R\subseteq 2\mathbb{Z}_{2n}+1$ and $T\subseteq 2\mathbb{Z}_{2n}$ are subsets of size $n/2$ such that  $R\cap (n+R)=T\cap (n+T)=\emptyset$ and $|R\cap (i+R)|+|T\cap (i+T)|=n/2$ for all $i\in 2\mathbb{Z}_{2n}\setminus\{0,n\}$.
 If $d\geq 6$, then   $d_{\frac{1}{2}X}\geq 3$. However, this is  impossible by Theorem \ref{thm::cir}, Theorem \ref{thm::dih} and Lemma \ref{lem::2.5}.

{\flushleft \bf Case C.} $X$ is bipartite but not antipodal.

By Lemma \ref{lem::2.1} and Corollary \ref{cor::3.1}, $\frac{1}{2}X$ is a primitive distance-regular circulant, dihedrant or dicirculant. Then it follows from Theorem \ref{thm::cir}, Theorem  \ref{thm::dih} and Lemma \ref{lem::2.5} that $\frac{1}{2}X$ is a complete graph, a cycle of prime order, or a Paley graph of prime order. We claim that the later two cases cannot occur, since  $\frac{1}{2}X$ has $2n$ vertices. Thus  $\frac{1}{2}X\cong K_{2n}$, and $d=2$ or $3$. If $d=2$, then $X$ is a complete bipartite graph, which is impossible because $X$ is non-trivial. Hence, $X$ is a   non-antipodal bipartite non-trivial distance-regular graph with diameter $3$. Recall that $X=\SDG=\Cay(\SD,\alpha^R\cup \alpha^T\beta)$ where  $R=-R$ and $T=(n+1)T$. Let $H$ be the bipartition set of $X$ containing the identity $1\in \SD$. Note that $H=\mathcal{N}_0\cup \mathcal{N}_2$. By Lemma \ref{lem::2.9}, $H$ is a subgroup of $\SD$ with index $2$.  Observe that  $\SD$ has exactly three  subgroups of index $2$, namely $\langle\rho\rangle$,  $\langle\rho^2,\tau\rangle$ and $\langle\rho^2,\rho\tau\rangle$. Thus we only need to consider the following three situations.
{\flushleft \bf Subcase C.1.} $H=\langle\rho\rangle$.

In this situation,  $R=\emptyset$. By Lemma \ref{lem::2.10},  $D=\tau\rho^T\tau=\rho^{(n-1)T}=\rho^{-T}$ is a non-trivial $(2n,k,\mu)$-difference set in the group $\langle\rho\rangle$, which clearly satisfies the condition  $D^{(-1)}=(\rho^{-T})^{(-1)}=\rho^{T}=\tau\rho^{-T}\tau=\tau D\tau$. Then $\rho^{T}$ is also a difference set in the group $\langle\rho\rangle$. As  $\langle\rho\rangle\cong \mathbb{Z}_{2n}$, we assert that $T$ is a non-trivial $(2n,k,\mu)$-difference set in the group $\mathbb{Z}_{2n}$.

{\flushleft \bf Subcase C.2.} $H=\langle\rho^2,\tau\rangle$.

In this situation, $R$ and $T$ are non-empty subsets of $2\mathbb{Z}_{2n}+1$.  By Lemma \ref{lem::2.10}, $D=\rho^{-1}(\rho^R\cup \rho^T\tau)=\rho^{-1+R}\cup\rho^{-1+T}\tau$ is a non-trivial $(2n,k,\mu)$-difference set in the group $\langle\rho^2,\tau\rangle$, which clearly satisfies the condition $D^{(-1)}=(\rho^{-1+R}\cup\rho^{-1+T}\tau)^{(-1)}=\rho^{1+R}\cup\tau \rho^{1-T}=\rho(\rho^{-1+R}\cup\rho^{-1+T}\tau)\rho=\rho D\rho$ because $T=(n+1)T=n+T$.  

{\flushleft \bf Subcase C.3.} $H=\langle\rho^2,\rho\tau\rangle$.

In this situation, $R$ and $T$ are non-empty subsets of $2\mathbb{Z}_{2n}+1$ and $2\mathbb{Z}_{2n}$, respectively.
By Lemma \ref{lem::2.10}, $D=\rho^{-1}(\rho^R\cup \rho^T\tau)=\rho^{-1+R}\cup\rho^{-1+T}\tau$ is a non-trivial $(2n,k,\mu)$-difference set in the group $\langle\rho^2,\rho\tau\rangle$, which clearly satisfies the condition $D^{(-1)}=(\rho^{-1+R}\cup\rho^{-1+T}\tau)^{(-1)}=\rho^{1+R}\cup\tau \rho^{1-T}=\rho(\rho^{-1+R}\cup\rho^{-1+T}\tau)\rho=\rho D\rho$ because $T=(n+1)T$.

This completes the proof.
\end{proof}

\renewcommand\proofname{\bf{Proof of Theorem \ref{thm::main2}}}
\begin{proof}

First of all, it is easy to see that the graphs in  (i)--(iii) are trivial distance-regular pseudo-semi-dihedrants.
Furthermore, by Lemma \ref{lem::2.10}, it is routine to verify that the graphs in (iv)--(vi) are non-antipodal bipartite non-trivial distance-regular pseudo-semi-dihedrants with diameter $3$. Also, by Lemma \ref{lem::3.9},  the graph in (vii) is a $2$-fold antipodal bipartite non-trivial distance-regular pseudo-semi-dihedrant with diameter $4$ (i.e., Hadamard graph).

Conversely, suppose that $X=\PSDG$ is a distance-regular pseudo-semi-dihedrant of diameter $d$. Clearly, $X$ cannot be a cycle  because $n>4$. If $X$ is trivial, then $X$ would be one of the graphs listed in (i)--(iii). If $X$ is non-trivial, then $k\geq 3$, and  Corollary \ref{cor::2.1} implies that  $X$ is imprimitive, and it suffices to consider the following three cases.

{\flushleft \bf Case A.} $X$ is antipodal but not bipartite.

By Lemma \ref{lem::2.1} and Corollary \ref{cor::3.2}, the antipodal quotient $\overline{X}$ of $X$ is a primitive distance-regular circulant or Cayley graph over $\mathbb{Z}_t\oplus\mathbb{Z}_2$ for some $t\mid n$. Then it follows from Theorem \ref{thm::cir} and Lemma \ref{lem::2.6} that $\overline{X}$ is a complete graph, a cycle of prime order, or a Paley graph of prime order.  If $\overline{X}$ is a cycle of prime order, then $X$ would be a cycle, which is impossible because $X$ is non-trivial. If $\overline{X}$ is a Paley graph of prime order, by Lemma \ref{lem::2.4}, we also can deduce a contradiction. Thus  $\overline{X}$ is a complete graph, and so $d=2$ or $3$ according to Lemma \ref{lem::2.1}. By Lemma \ref{lem::3.7}, $d\neq 3$, whence $d=2$. However, complete multipartite graphs are the only antipodal distance-regular graphs with diameter $2$. Therefore, there are no non-trivial distance-regular pseudo-semi-dihedrants which are antipodal but not bipartite.

{\flushleft \bf Case B.} $X$ is antipodal and bipartite.

If $d$ is odd, by Lemma \ref{lem::2.1}, $\overline{X}$ is primitive. Also, by  Corollary \ref{cor::3.2},  $\overline{X}$ is a distance-regular circulant or Cayley graph over $\mathbb{Z}_t\oplus\mathbb{Z}_2$ for some $t\mid n$. 
As in Case A, we assert that $\overline{X}$ is a complete graph. Hence, $d=3$. Considering that $X$ is antipodal and bipartite, we obtain  $X\cong K_{2n,2n}-2nK_2$, which is impossible because $X$ is non-trivial.
Now suppose that $d$ is even. Then Corollary \ref{cor::3.2} and Lemma \ref{lem::2.1} indicate that $\frac{1}{2}X$ is an antipodal non-bipartite distance-regular circulant or Cayley graph over $2\mathbb{Z}_{2n}\oplus\mathbb{Z}_2\cong \mathbb{Z}_n\oplus\mathbb{Z}_2$ with diameter $d_{\frac{1}{2}X}=d/2$.  Clearly, $d\neq 2$. If $d=4$, by Lemma \ref{lem::3.9}, we conclude that 
$R=-R$ and $T=(n-1)T=n-T$ are subsets of $2\mathbb{Z}_{2n}+1$ of size $n/2$ such that  $R\cap (n+R)=T\cap (n+T)=\emptyset$ and $|R\cap (i+R)|+|T\cap (i+T)|=n/2$ for all $i\in 2\mathbb{Z}_{2n}\setminus\{0,n\}$. If $d\geq 6$, then  $d_{\frac{1}{2}X}\geq 3$.  However, this is impossible by Theorem \ref{thm::cir} and Lemma \ref{lem::2.6}.

{\flushleft \bf Case C.} $X$ is bipartite but not antipodal.

By Lemma \ref{lem::2.1} and Corollary \ref{cor::3.2}, $\frac{1}{2}X$ is a primitive distance-regular circulant or Cayley graph over $2\mathbb{Z}_{2n}\oplus\mathbb{Z}_2\cong \mathbb{Z}_n\oplus\mathbb{Z}_2$. As in Case A,  $\frac{1}{2}X$ is a complete graph, a cycle of prime order, or a Paley graph of prime order. We claim that the latter two cases cannot occur, since  $\frac{1}{2}X$ has $2n$ vertices. Thus  $\frac{1}{2}X\cong K_{2n}$, and $d=2$ or $3$. If $d=2$, then $X$ is a complete bipartite graph, which is impossible because $X$ is non-trivial. Hence, $X$ is a   non-antipodal bipartite non-trivial distance-regular graph with diameter $3$. Recall that $X=\PSDG=\Cay(\PSD,\rho^R\cup \rho^T\tau)$ where  $R=-R$ and $T=(n-1)T$. Let $H$ be the bipartition set of $X$ containing the identity $1\in \PSD$. Note that $H=\mathcal{N}_0\cup \mathcal{N}_2$. By Lemma \ref{lem::2.9}, $H$ is a subgroup of $\PSD$ with index $2$.  Observe that  $\PSD$ has exactly three  subgroups of index $2$, namely $\langle\rho\rangle$,  $\langle\rho^2,\tau\rangle=\langle\rho^2\rangle\times \langle\tau\rangle$ and $\langle\rho^2,\rho\tau\rangle=\langle\rho\tau\rangle$. Thus we only need to consider the following three situations.

{\flushleft \bf Subcase C.1.} $H=\langle\rho\rangle$.

In this situation,  $R=\emptyset$. By Lemma \ref{lem::2.10},  $D=\tau\rho^T\tau=\rho^{-T}$ is a non-trivial $(2n,k,\mu)$-difference set in the group $\langle\rho\rangle$, which clearly satisfies the condition  $D^{(-1)}=(\rho^{-T})^{(-1)}=\rho^{T}=\tau\rho^{-T}\tau=\tau D\tau$ because $T=(n-1)T$. Then $\rho^{T}$ is also a difference set in the group $\langle\rho\rangle$. As  $\langle\rho\rangle\cong \mathbb{Z}_{2n}$, we assert that $T$ is a non-trivial $(2n,k,\mu)$-difference set in the group $\mathbb{Z}_{2n}$.

{\flushleft \bf Subcase C.2.} $H=\langle\rho^2,\tau\rangle=\langle\rho^2\rangle\times \langle\tau\rangle$.

In this situation, $R$ and $T$ are non-empty subsets of $2\mathbb{Z}_{2n}+1$. By Lemma \ref{lem::2.10}, $D=\rho^{-1}(\rho^R\cup \rho^T\tau)=\rho^{-1+R}\cup\rho^{-1+T}\tau$ is a non-trivial $(2n,k,\mu)$-difference set in the group $\langle\rho^2,\tau\rangle =\langle\rho^2\rangle\times \langle\tau\rangle$, which clearly satisfies the condition $D^{(-1)}=(\rho^{-1+R}\cup\rho^{-1+T}\tau)^{(-1)}=\rho^{1+R}\cup\tau \rho^{1-T}=\rho(\rho^{-1+R}\cup\rho^{-1+T}\tau)\rho=\rho D\rho$ because $T=(n-1)T=n-T$. As $\langle\rho^2,\tau\rangle=\langle\rho^2\rangle\times \langle\tau\rangle\cong 2\mathbb{Z}_{2n}\oplus\mathbb{Z}_2$, we assert that $(-1+R,0)\cup(-1+T,1)$ is a non-trivial $(2n,k,\mu)$-difference set in the group $2\mathbb{Z}_{2n}\oplus\mathbb{Z}_2$.

{\flushleft \bf Subcase C.3.} $H=\langle\rho^2,\rho\tau\rangle=\langle\rho\tau\rangle$.

In this situation, $R$ and $T$ are non-empty subsets of $2\mathbb{Z}_{2n}+1$ and  $2\mathbb{Z}_{2n}$, respectively. By Lemma \ref{lem::2.10}, $D=\rho^{-1}(\rho^R\cup \rho^T\tau)=\rho^{-1+R}\cup\rho^{-1+T}\tau$ is a non-trivial $(2n,k,\mu)$-difference set in the group $\langle\rho^2,\rho\tau\rangle=\langle\rho\tau\rangle$, which clearly satisfies the condition $D^{(-1)}=(\rho^{-1+R}\cup\rho^{-1+T}\tau)^{(-1)}=\rho^{1+R}\cup\tau \rho^{1-T}=\rho(\rho^{-1+R}\cup\rho^{-1+T}\tau)\rho=\rho D\rho$ because $T=(n-1)T=-T$.

This completes the proof.
\end{proof}

\section{Further research}
In this paper, we provide a partial characterization for distance-regular Cayley graphs over semi-dihedral and pseudo-semi-dihedral groups. These  groups are two special classes of $p$-groups with a cyclic subgroup of index $p$. According to elementary group theory, it is known that every group of order $p^s$ with a cyclic subgroup of index $p$ must be one of the following groups:
\begin{enumerate}[(i)]
\setlength{\itemsep}{0pt}
\item the cyclic group $\langle \rho\mid \rho^{p^s}=1\rangle$, where $s\geq 1$;
\item the abelian group $\langle \rho,\tau\mid \rho^{p^{s-1}}=\tau^p=1,\rho\tau=\tau\rho\rangle$, where $s\geq 2$; 
\item the metacyclic group $\langle \rho,\tau\mid \rho^{p^{s-1}}=\tau^p=1,\tau^{-1}\rho\tau=\rho^{p^{s-2}+1}\rangle$, where $p>2$ and $s\geq 3$;
\item the dihedral group $\langle \rho,\tau\mid \rho^{2^{s-1}}=\tau^2=1,\tau\rho\tau=\rho^{-1}\rangle$, where $s\geq 3$;
\item the dicyclic group $\langle\rho,\tau\mid \rho^{2^{s-1}}=1,\tau^2=\rho^{2^{s-2}},\tau^{-1}\rho\tau=\rho^{-1}\rangle$, where $s\geq 3$;
\item the semi-dihedral group $\langle\rho,\tau\mid\rho^{2^{s-1}}=\tau^2=1,\tau\rho\tau=\rho^{2^{s-2}-1} \rangle$, where $s\geq 4$;
\item the pseudo-semi-dihedral group $\langle\rho,\tau\mid\rho^{2^{s-1}}=\tau^2=1,\tau\rho\tau=\rho^{2^{s-2}+1} \rangle$, where $s\geq 4$.
\end{enumerate}

Up to now, the problem of characterizing  distance-regular Cayley graphs over the groups listed in  (i), (ii), (iv)--(vii) has been considered in \cite{HD22,MP03,MP07,ZLH23} and in the current paper. Thus we propose the following problem for further research.

\begin{problem}
Determine all distance-regular Cayley graphs over the metacyclic group $\langle \rho,\tau\mid \rho^{p^{s-1}}=\tau^p=1,\tau^{-1}\rho\tau=\rho^{p^{s-2}+1}\rangle$, where $p>2$ and $s\geq 3$.
\end{problem}

\section*{Acknowledgments}
X. Huang is supported by National Natural Science Foundation of China (Grant No. 11901540). L. Lu is supported by National Natural Science Foundation of China (Grant Nos. 12371362, 12001544) and  Natural Science Foundation of Hunan Province  (Grant No. 2021JJ40707).

\end{document}